\documentclass{amsart}
\usepackage{graphicx} 
\usepackage[utf8]{inputenc}
\usepackage{bm}
\usepackage{relsize}
\usepackage{verbatim}
\usepackage{physics}
\usepackage{hyperref}
\usepackage{amssymb}
\usepackage{amsmath}
\usepackage{amsthm}
\usepackage{enumitem}
\usepackage{tikz-cd}
\usepackage{nicefrac}
\usepackage{faktor}
\usetikzlibrary{matrix,arrows,decorations.pathmorphing}
\usepackage[a4paper, left=3cm, right=3cm, top=2cm]{geometry}
\usepackage{mdframed}
\usepackage{lipsum}
\usepackage{enumitem}
\usepackage{mathrsfs}

\numberwithin{equation}{section}

\theoremstyle{definition}
\newtheorem{thm}{Theorem}[section] 
\newtheorem{lemma}[thm]{Lemma}
\newtheorem{proposition}[thm]{Proposition}
\newtheorem{cor}[thm]{Corollary}

\newtheorem{defn}[thm]{Definition} 

\title{Small solutions to linear forms in primes}
\author{Tammo Dede}

\newcommand{\NN}{\mathbb{N}}
\newcommand{\ZZ}{\mathbb{Z}}
\newcommand{\RR}{\mathbb{R}}

\newcommand{\LL}{\mathbb{L}}

\newcommand{\SSS}{\mathfrak{S}}
\newcommand{\bfa}{\textbf{a}}
\newcommand{\bx}{\textbf{x}}
\newcommand{\by}{\textbf{y}}
\newcommand{\bv}{\textbf{v}}
\newcommand{\bc}{\textbf{c}}
\newcommand{\bd}{\textbf{d}}
\newcommand{\bw}{\textbf{w}}
\newcommand{\bz}{\textbf{z}}
\newcommand{\bp}{\textbf{p}}
\newcommand{\Pri}{\mathcal{P}}
\newcommand{\Ball}{\mathcal{B}}
\newcommand{\prim}{\text{prim}}
\newcommand{\loc}{\text{loc}}
\newcommand{\mix}{\text{mix}}
\newcommand{\vol}{\text{vol}}
\newcommand{\Jfrak}{\mathfrak{J}}
\newcommand{\Sfrak}{\mathfrak{S}}
\newcommand{\Ccal}{\mathcal{C}}
\newcommand{\Ecal}{\mathcal{E}}
\newcommand{\Gcal}{\mathcal{G}}
\newcommand{\Rcal}{\mathcal{R}}
\newcommand{\Ncal}{\mathcal{N}}
\newcommand{\Vcal}{\mathcal{V}}
\newcommand{\Tcal}{\mathcal{T}}
\newcommand{\Scal}{\mathcal{S}}
\newcommand{\Jcal}{\mathcal{J}}
\newcommand{\Ical}{\mathcal{I}}
\newcommand{\bald}{\textbf{b}}
\newcommand{\buld}{\textbf{u}}
\newcommand{\bvld}{\textbf{v}}
\newcommand{\bolt}{\textbf{t}}
\newcommand{\Rfrak}{\mathfrak{R}}
\newcommand{\dfrak}{\mathfrak{d}}
\newcommand{\Span}{\text{Span}}
\newcommand{\mm}{\mathfrak{m}}

\begin{document}
\begin{abstract}
    We show that a positive proportion of linear forms in four variables admit a solution in the primes that is as small as one would heuristically expect. Out of the linear forms that satisfy certain local solvability conditions, almost all admit small prime solutions.
\end{abstract}

\maketitle
\section{Introduction}$~$

Let $\bfa=(a_1,a_2,a_3,a_4)$ be a tuple of non-zero integers such that $\gcd(a_1;a_2;a_3;a_4)=1$. We call such tuples \textit{primitive} and denote the set of these by $\ZZ^4_{\prim}$. In this article we investigate small prime solutions $p_1,p_2,p_3,p_4$ to equations of the form 
\begin{equation}\label{Equation 1}
    a_1p_1+a_2p_2+a_3p_3+a_4p_4=0.
\end{equation}

Given a primitive integer tuple $\bfa$, we call it \textit{locally solvable} if the coefficients of $\bfa$ are not all of the same sign and if for every prime $p$ there exists a solution to Equation (\ref{Equation 1}) in the reduced residues modulo $p$. The set of locally solvable integer tuples will be denoted by $\LL^{loc}$. The goal of this article is to give a bound on the smallest prime solution to Equation (\ref{Equation 1}) in terms of the size of $\bfa$ for most $\bfa\in\LL^{\loc}$ following the approach of Browning, Le-Boudec and Sawin \cite{BrLBSa}. The adjustments for this setting are similar to work of Holdridge \cite{Hol}.

Writing $|\bx|=\max|x_i|$, one can use the circle method together with the prime number theorem to show that the box of prime tuples $\bp=(p_1,p_2,p_3,p_4)$ with $|\bp|\leq B$ contains
\begin{equation*}
    C_{\bfa}\frac{B^3}{(\log B)^4}\big(1+o_{\bfa}(1)\big)
\end{equation*}
many solutions to Equation (\ref{Equation 1}), where $C_{\bfa}>0$ if $\bfa$ is locally solvable. While this approach is sufficient to establish the existence of many solutions, it does not produce a satisfactory bound on the smallest one. In a naive approach one has to choose $B$ of exponential size in $|\bfa|$. Given any primitive tuple $\bfa$, we denote the set of prime solutions to (\ref{Equation 1}) by $L(\bfa)$ and let
\begin{equation*}
    \mm(\bfa)=\begin{cases}
        \min_{\bp\in L(\bfa)}|\bp|&\text{if }L(\bfa)\neq\emptyset,\\
        \infty&\text{else.}
    \end{cases}
\end{equation*}
For $\bfa\in\LL^{\loc}$, the approach above produces the bound $\mm(\bfa)\ll\exp(O(|\bfa|))$. There is work of Liu and Tsang \cite{LiuTsang} on the similar problem of finding a small prime solution $p_1,p_2,p_3$ to the linear equation
\begin{equation}\label{Equation b}
    a_1p_1+a_2p_2+a_3p_3=b
\end{equation}
for sufficiently large integers $b$ satisfying some congruence conditions. They were able to show the existence of an effective constant $c>0$ such that there is a prime solution $\bp$ to Equation (\ref{Equation b}) of size $|\bp|\leq 3|b|+|\bfa|^c$, making use of a refinement of the circle method to deal with multiple main terms. Moreover they needed sharp estimates on the distribution of zeros of certain Dirichlet $L$-functions. While this bound is polynomial in the size of the coefficients, it does not match their conjectured bound of $c=2+\epsilon$.

In the setting of Equation (\ref{Equation 1}) there are roughly $B^4/(\log B)^4$ many tuples of primes up to size $B$. Heuristically one could hope that the values on the left hand side of (\ref{Equation 1}) equidistribute in the interval $[-4|\bfa|B,4|\bfa|B]$, thus obtaining a solution to Equation (\ref{Equation 1}) as soon as $B^4/(\log B)^4\geq 4|\bfa|B$. Therefore we could hope for the estimate
\begin{equation*}
    \mm(\bfa)\ll_{\epsilon}|\bfa|^{1/3+\epsilon}
\end{equation*}
for $\epsilon>0$.

Instead of making refinements to the circle method approach, we switch to a statistical approach and make use of averaging over $\bfa$. For this, let $\|.\|$ denote the Euclidean norm, let $A>1$ be some real parameter, and set
\begin{equation*}
    \LL(A)=\{\bfa\in\ZZ^4_{\prim}:\|\bfa||\leq A\}.
\end{equation*}
Moreover, we let
\begin{equation*}
    \LL^{\loc}(A)=\{\bfa\in\LL^{\loc}:\|\bfa\|\leq A\}
\end{equation*}
the set of locally solvable tuples. We investigate the ratio
\begin{equation*}
    \varrho(A)=\frac{\#\big\{\bfa\in\LL^{\loc}(A):\mm(\bfa)^3\leq|\bfa|(\log|\bfa|)^4(\log\log|\bfa|)\big\}}{\#\LL^{\loc}(A)}
\end{equation*}
and we will show the following.
\begin{thm}\label{theorem:main_theorem}
    In the notation above, we have
    \begin{equation*}
        \lim_{A\rightarrow\infty}\varrho(A)=1.
    \end{equation*}
\end{thm}
Furthermore, we will also compute that it is not rare for elements of $\LL(A)$ to lie in $\LL^{\loc}(A)$.
\begin{thm}\label{theorem:L^loc_density_estimate}
We have
\begin{equation*}
    \liminf_{A\rightarrow\infty}\frac{\#\LL^{\loc}(A)}{\#\LL(A)}>0.
\end{equation*}
\end{thm}
This directly gives us the following result.
\begin{cor}
    We have
    \begin{equation*}
        \liminf_{A\rightarrow\infty}\frac{\#\big\{\bfa\in\LL^{\loc}(A):\mm(\bfa)^3\leq|\bfa|(\log|\bfa|)^4(\log\log|\bfa|)\big\}}{\#\LL(A)}>0.
    \end{equation*}
\end{cor}
We will prove Theorem \ref{theorem:L^loc_density_estimate} in Section \ref{section:density_theorem_proof} by a straightforward computation, which is mostly independent of the rest of this article. To establish Theorem \ref{theorem:main_theorem}, we employ an upper bound on a variance comparing a counting function for solutions to Equation (\ref{Equation 1}) to its expected value, which will be stated in Section \ref{section:outline_of_proof}.

In a more general setting, if $s\geq 5$ is an integer and $(a_1,\dots,a_s)\in(\ZZ\setminus\{0\})^s$ is such that $\gcd(a_1;\dots;a_s)=1$, we can consider equations of the form
\begin{equation}\label{Equation 2}
    a_1p_1+a_2p_2+\cdots+a_sp_s=0
\end{equation}
with primes $p_1,\dots,p_s$. Assuming some congruence conditions on the $a_i$, one can prove similar results to Theorem \ref{theorem:main_theorem}, finding solutions of size $|\bfa|^{1/(s-1)}(\log|\bfa|)^c$ for some $c>0$ on average. This can be done by essentially following the work of Brüdern and Dietmann \cite{BRUDERN201418}, establishing a result similar to Proposition \ref{prop:variance_upper_bound} via a direct approach with the circle method.

\section{Outline of the proof}\label{section:outline_of_proof}$~$
Let $\langle\cdot,\cdot\rangle$ denote the Euclidean inner product on $\RR^4$. We start this section by introducing a counting function for solutions of Equation (\ref{Equation 1}). Let $\Pri$ denote the set of primes and let $B\geq 1$ be some real parameter. We set
\begin{equation}\label{defn:P(B)}
    \Pri(B)=\{\bp=(p_1,\dots,p_4):|\bp|\leq B,\,p_i\in\Pri\}.
\end{equation}
For $\bc\in\ZZ^4$, we define the lattice
\begin{equation}\label{defn:Lambda_c lattice}
    \Lambda_{\bc}=\big\{\by\in\ZZ^4:\langle\bc,\by\rangle=0\big\}
\end{equation}
and let
\begin{equation}\label{defn:N_a(B)}
    N_{\bfa}(B)=(\log B)^4\sum_{\substack{\bx\in\Pri(B)\\\bfa\in\Lambda_{\bx}}}1.
\end{equation}
We will follow the strategy of Browning, Le-Boudec and Sawin to show that $N_{\bfa}(B)$ is on average well approximated by a local counting function. Our local counting function will be an adjusted version of the one used in \cite{BrLBSa} that closely resembles the version of Holdridge \cite{Hol}. We then show that this local counting function is large enough on average to obtain a positive number of solutions to Equation (\ref{Equation 1}).

By a heuristic argument, which can be made precise with a circle method, we expect $N_{\bfa}(B)$ to grow like $B^3  |\bfa|^{-1}$ for $B$ large. The local counting function will imitate this behavior. For any integer $N\geq 1$, any real $\gamma >0$ and $\bv\in\RR^N$, we define the region
\begin{equation}\label{defn:C^gamma}
    \Ccal_{\bv}^{(\gamma)}=\Big\{\bolt\in\RR^N:|\langle\bv,\bolt\rangle|\leq\frac{\|\bv\|\cdot\|\bolt\|}{2\gamma}\Big\}.
\end{equation}
For any integer $Q\geq 1$ and $\bc\in\ZZ^N$, we define the lattice
\begin{equation}\label{defn:lambda_c(Q) lattice}
    \Lambda_{\bc}^{(Q)}=\big\{\by\in\ZZ^N:\langle\bc,\by\rangle\equiv 0 \text{ mod }Q\big\}.
\end{equation}
Further we let
\begin{equation}\label{defn:alpha}
    \alpha=\log B.
\end{equation}
At this point we have to use the adjustments from Holdridge's version. For $B$ sufficiently large, we set
\begin{equation}\label{def:w}
    w=\frac{\log\log B}{\log\log\log B}
\end{equation}
and
\begin{equation}\label{def:W}
    W=\prod_{p\leq w}p^{\lceil\log w/\log p\rceil+1}.
\end{equation}
Here we need $W$ to not exceed some power of $\log B$ in size. In our case, an application of the estimate $\#\{p\leq w:p\text{ prime}\}\ll w/(\log w)$ to $\log W$ directly yields the bound
\begin{equation}\label{W_bound}
    W\ll (\log B)^3.
\end{equation}
This is sufficient for our purposes. We are now set up to introduce the local counting function. Let
\begin{equation}\label{defn:N^loc}
    N_{\bfa}^{\loc}(B)=(\log B)^4\frac{\alpha W}{\|\bfa\|}\sum_{\substack{\bx\in\Pri(B)\\\bfa\in\Lambda_{\bx}^{(W)}\cap\Ccal^{(\alpha)}_{\bx}}}\frac{1}{\|\bx\|}.
\end{equation}
As in \cite[(2.11)]{BrLBSa}, we expect a vector $\bx$ to satisfy $\bfa\in\Lambda_{\bx}^{(W)}\cap\Ccal^{(\alpha)}_{\bx}$ with probability $(\alpha W)^{-1}$. Therefore we expect $N_{\bfa}^{\loc}(B)$ to be of size $B^3\|\bfa\|^{-1}$ on average, which matches the expectation for $N_{\bfa}(B)$. The following proposition quantifies how often $N_{\bfa}(B)$ is well approximated by $N_{\bfa}^{\loc}(B)$ and is the main ingredient for proving Theorem \ref{theorem:main_theorem}.
\begin{proposition}\label{prop:variance_upper_bound}
    Let $A$ and $B\geq 3$ be positive real numbers such that
    \begin{equation*}
        B^2(\log B)<A\leq B^3(\log B)^{-4}(\log\log B)^{-1}.
    \end{equation*}
    Then there exists a $C>0$ such that we have
    \begin{equation*}
        \sum_{\bfa\in\LL(A)}\big|N_{\bfa}(B)-N_{\bfa}^{\loc}(B)\big|^2\ll\frac{A^2B^6}{(\log\log B)^C}.
    \end{equation*}
\end{proposition}
The proof of this proposition is done in Section \ref{section:proof_of_variance_bound} and draws heavily from results about counting lattice points in certain regions. The preparatory work for that is done in Section \ref{section:geometry_of_numbers}. A direct consequence of this proposition is the following result.
\begin{cor}\label{corollary:number_of_bad_approx}
    Let $A$, $B$ and $C>0$ be as in the previous proposition. For $C/3>\delta>0$ we have
    \begin{equation*}
        \#\Big\{\bfa\in\LL(A):\big|N_{\bfa}(B)-N_{\bfa}^{\loc}(B)\big|>\frac{B^3}{|\bfa|(\log\log B)^{\delta}}\Big\}\ll \frac{A^4}{(\log\log A)^{\delta}}.
    \end{equation*}
\end{cor}
We combine this with a lower bound on $N_{\bfa}^{\loc}(B)$. This way we will be able to produce small solutions to Equation \ref{Equation 1}. The following result will be proven in Section \ref{section:N^loc_lower_bound}.
\begin{proposition}\label{prop:number_of_bad_N^loc}
    Let $A,B$ be two positive real numbers with $B^2\leq A$. Let $\kappa>0$, we have
    \begin{equation*}
        \#\Big\{\bfa\in\LL^{\loc}(A):N_{\bfa}^{\loc}(B)\leq \frac{B^3}{A(\log\log A)^{\kappa}}\Big\}\ll\frac{A^4}{(\log\log A)^{\kappa/3}}.
    \end{equation*}
\end{proposition}
We close out this section by proving Theorem \ref{theorem:main_theorem} from the above results.
\begin{proof}[Proof of Theorem \ref{theorem:main_theorem}]
    Let $B$ be sufficiently large and pick $A= B^3(\log B)^{-4}(\log\log B)^{-1}$. Let $\bfa\in\LL^{\loc}(A)$ of size $\frac{1}{2}A<\|\bfa\|\leq A$. Note that we have $|\bfa|\leq\|\bfa\|\leq 2|\bfa|$. If $\bfa$ is neither counted in Corollary \ref{corollary:number_of_bad_approx} nor in Proposition \ref{prop:number_of_bad_N^loc} for some $\kappa<\delta$, then we have that
    \begin{align*}
        N_{\bfa}(B)\geq N_{\bfa}^{\loc}(B)-\frac{B^3}{|\bfa|(\log\log B)^{\delta}}\gg \frac{B^3}{|\bfa|(\log\log B)^{\kappa}}.
    \end{align*}
    In that case it is $N_{\bfa}(B)>0$. By our choice of $A$, we find that $B(\log B^2)^{-4/3}(\log\log B)^{-1/3}<A^{1/3}$. Now for $B$ large enough, we have $A>2B^2$ trivially, hence $\log B^2<(\log A/2)\leq\log |\bfa|$. Similarly we get that $\log\log B < \log\log|\bfa|$. Therefore we find that
    \begin{equation*}
        B\leq\big(|\bfa|(\log|\bfa|)^4(\log\log|\bfa|)\big)^{1/3}.
    \end{equation*}
    In the remaining cases, $\bfa$ is either counted in Corollary \ref{corollary:number_of_bad_approx} or Proposition $\ref{prop:number_of_bad_N^loc}$. However, there exists $c>0$ such that there are at most $O(A^4(\log\log A)^{-c})$
    of these $\bfa$. By Theorem \ref{theorem:L^loc_density_estimate} this set is negligible. We now sum this argument over dyadic pieces, considering $2^{-j}A<\|\bfa\|\leq 2^{-j+1}A$, to obtain the theorem.
\end{proof}

\section*{Notation}$~$
As already done in the introductory sections, for two real functions $f$ and $g$ with $g(x)>0$ we write $f=O(g)$ or $f\ll g$ if $|f(x)|\leq Cg(x)$ for some $C>0$ as $x\rightarrow\infty$. For any real $\alpha$, we let
\begin{equation*}
    e(\alpha)=\exp(2\pi i\alpha ).
\end{equation*}
If $q\geq 1$ is an integer, we indicate sums running over a complete set of residue classes $a$ modulo $q$ by $\sum_{a\,(q)}$. For sums over a complete set of reduced residue classes, we write
\begin{equation*}
    \sideset{}{^*}\sum_{a\,(q)}.
\end{equation*}
For integers $a_1,a_2,\dots,a_n$, the object $(a_1,a_2,\dots,a_n)$ is reserved for the element in $\ZZ^n$. We denote the gcd of $a_1,a_2,\dots,a_n$ by $\gcd(a_1;a_2;\dots;a_n)$ or $(a_1;a_2;\ldots;a_n)$.

\section{The geometry of numbers}\label{section:geometry_of_numbers}$~$
In this section we recall some needed results from the geometry of numbers. Most of the following results are covered in Section 3 of \cite{BrLBSa} but we write them down for completeness. Let $N\geq 1$ be an integer. We call a discrete subgroup $\Lambda$ of $\RR^N$ a \textit{lattice}. Its rank is the dimension of the subspace $\text{Span}_{\RR}(\Lambda)$ of $\RR^N$. For a lattice $\Lambda$ of rank $R\geq 1$ with a basis $(\bald_1,\dots,\bald_R)$, we let $\textbf{B}$ be the $N\times R$ matrix with columns $\bald_i$ for $1\leq i\leq R$. The \textit{determinant} of $\Lambda$ is then given by
\begin{equation}\label{def:det of a lattice}
    \det(\Lambda)=\sqrt{\det\big(\textbf{B}^T\textbf{B}\big)}.
\end{equation}
This definition is independent of the base chosen. If moreover $\Lambda\subset\ZZ^N$, then we call $\Lambda$ an \textit{integral lattice}. We say that an integral lattice of rank $R$ is \textit{primitive} if it is not properly contained in another integral lattice of rank $R$, which is the case if and only if $\text{Span}_{\RR}(\Lambda)\cap\ZZ^N=\Lambda$.
For any real $u>0$ we let
\begin{equation*}
    \Ball_N(u)=\{\by\in\RR^N:\|\by\|< u\},
\end{equation*}
and for any integer $R\geq 1$ we let
\begin{equation}\label{def:V_N}
    V_R=\vol_{\RR^R}\big(\Ball_R(1)\big).
\end{equation}
\begin{defn}
    Let $N\geq 1$ be an integer and $R\in\{1,\dots,N\}$. Given a lattice $\Lambda\subset\RR^N$ of rank $R$, for $1\leq i\leq R$ we define
    \begin{equation*}
        \lambda_i(\Lambda)=\inf\big\{u\in\RR_{>0}:\dim\big(\text{Span}_{\RR}\big(\Lambda\cap\Ball_N(u)\big)\big)\geq i\big\}.
    \end{equation*}
    We call $\lambda_i(\Lambda)$ the $i-$th successive minimum of $\Lambda$.
\end{defn}
We have that $\lambda_1(\Lambda)\leq\lambda_2(\Lambda)\leq\dots\leq\lambda_R(\Lambda)$ and by Minkowski's second theorem (see for example \cite{cassels2012introduction}) it is
\begin{equation*}
    \det(\Lambda)\leq\lambda_1(\Lambda)\cdots\lambda_R(\Lambda)\ll\det(\Lambda),
\end{equation*}
where the implied constant depends at most on the dimension $R$. We will need to be able to count lattice points in certain regions. If the region is large enough depending on the successive minima, the following lemma supplies an asymptotic formula for the number of lattice points in that region. This is Lemma 3.5 in \cite{BrLBSa}. 
\begin{lemma}\label{Lemma:General point count}
    Let $N\geq 2$ be an integer and $R\in\{1,\dots,N\}$. Let $\Lambda\subset\RR^N$ be a lattice of rank $R$. Further, let $I\in\{1,\dots,N-1\}$ and $\bvld_1,\dots,\bvld_I\in\RR^N$. For $T>0$ and $\gamma >0$ we define
    \begin{equation*}
        \Rcal_{\bvld_1,\dots,\bvld_I}(T,\gamma)=\Ball_N(T)\cap\Ccal_{\bvld_1}^{(\gamma)}\cap\dots\cap\Ccal_{\bvld_I}^{(\gamma)}
    \end{equation*}
    and
    \begin{equation*}
        \Vcal_{\bvld_1,\dots,\bvld_I}(\Lambda;\gamma)=\vol\big(\Span_{\RR}(\Lambda)\cap\Rcal_{\bvld_1,\dots,\bvld_I}(1,\gamma)\big).
    \end{equation*}
    Let $Y\geq \lambda_R(\Lambda)$. For $T\geq Y$ we have
    \begin{equation*}
        \#\big(\Lambda\cap\Rcal_{\bvld_1,\dots,\bvld_I}(T,\gamma)\big)= \frac{T^R}{\det(\Lambda)}\Big(\Vcal_{\bvld_1,\dots,\bvld_I}(\Lambda;\gamma)+O\Big(\frac{Y}{T}\Big)\Big),
    \end{equation*}
    where the implied constant depends at most on $R$.
\end{lemma}
In the cases where the region is too small to obtain an asymptotic formula for the number of lattice points within that region, we can still obtain an upper bound for the number of points. The following lemma, which is \cite[Lemma 3.7]{BrLBSa}, deals with that. A proof can be found at the same place.
\begin{lemma}\label{Lemma:Degenerate point count}
    Let $N\geq 1$ be an integer and $R\in\{1,\dots,N\}$. Let $\Lambda\subset\RR^N$ be a lattice of rank $R$. Let $M>0$ be such that $M<\lambda_1(\Lambda)$ and let $Y\geq\lambda_R(\Lambda)$. For any $R_0\in\{0,\dots,R-1\}$ and $T\leq Y$ we have
    \begin{equation*}
        \#\big(\Lambda\setminus\{0\}\cap\Ball_N(T)\big)\ll \frac{T^{R-R_0}Y^{R_0}}{\det(\Lambda)}+\Big(\frac{T}{M}\Big)^{R-R_0-1}.
    \end{equation*}
    Further, let $j_0\in\{1,\dots,R-1\}$ and $J\geq M$ be such that $J<\lambda_{j_0+1}(\Lambda)$. For any $R_0\in\{0,\dots,R-1-j_0\}$ and $T\leq Y$ we have
    \begin{equation*}
        \#\big(\Lambda\setminus\{0\}\cap\Ball_N(T)\big)\ll \frac{T^{R-R_0}Y^{R_0}}{\det(\Lambda)}+\Big(\frac{T}{M}\Big)^{j_0}\Bigg(\Big(\frac{T}{J}\Big)^{R-R_0-1-j_0}+1\Bigg).
    \end{equation*}
\end{lemma}
If the lattice $\Lambda\subset\RR^N$ is an integral lattice, we might also ask for the number of \textit{primitive points} in a given region, that is, the points $\bx=(x_1,\dots,x_N)\in\Lambda\cap\ZZ^N$ such that $\gcd(x_1,\dots,x_N)=1$. The following lemma deals with this for primitive integral lattices if we fix the region to be $\Ball_N(T)$ for some $T$ large enough. This is a slightly adjusted version of \cite[Lemma 3]{LBpaper}.
\begin{lemma}\label{lemma:primitive lattice points}
    Let $N\geq 1$ be an integer and let $R\in\{2,\dots,N\}$. Let $\Lambda\subset\RR^N$ be a primitive integral lattice of rank $R$ and let $Y\geq\lambda_R(\Lambda)$. For real $T\geq Y$ we have
    \begin{equation*}
        \#\big(\Lambda\cap\ZZ_{\prim}^N\cap\Ball_N(T)\big)=\frac{V_R}{\zeta(R)}\frac{T^R}{\det(\Lambda)}\Big(1+O\Big(\frac{Y\log T}{T}\Big)\Big)+O\big(T\big).
    \end{equation*}
    The implied constants depend at most on $R$.
\end{lemma}
\begin{proof}
    We can mostly follow the proof of \cite[Lemma 3]{LBpaper} with a slight variation. To allow for lattices of rank 2, we need to pick up a factor of $\log Y$ at multiple places. We start by applying a Möbius inversion. For any integer $l\geq 1$ we have that $l^{-1}\Lambda\cap\ZZ^N=\Lambda$ since $\Lambda$ is primitive. With this in  mind, it is
    \begin{equation*}
        \#\big(\Lambda\cap\ZZ_{\prim}^N\cap\Ball_N(T)\big)=\sum_{l\leq T}\mu(l)\big(\#\big(\Lambda\cap\Ball_N(T/l)\big)-1\big).
    \end{equation*}
    We split the sum over $l$ into the parts where $1\leq l\leq T/Y$ and $T/Y<l\leq T$ so that we arrive at
    \begin{multline*}
        \#\big(\Lambda\cap\ZZ_{\prim}^N\cap\Ball_N(T)\big)=\sum_{l\leq T/Y}\mu(l)\#\big(\Lambda\cap\Ball_N(T/l)\big)+ \sum_{T/Y<l\leq T}\mu(l)\#\big(\Lambda\cap\Ball_N(T/l)\big)+O\big(T\big).
    \end{multline*}
    The first sum can be dealt with via \cite[Lemma 1]{LBpaper}. This yields
    \begin{equation*}
        \sum_{l\leq T/Y}\mu(l)\#\big(\Lambda\cap\Ball_N(T/l)\big)=V_R\sum_{l\leq T/Y}\frac{\mu(l)}{l^R}\frac{T^R}{\det(\Lambda)}\Big(1+O\Big(\frac{Yl}{T}\Big)\Big).
    \end{equation*}
    Using that $R\geq 2$, we see that the sum over the error term is bounded by $O\big(T^{R-1}\det(\Lambda)^{-1}Y\log T\big)$. For the main term we note that
    \begin{equation*}
        \sum_{l>T/Y}\frac{1}{l^R}\frac{T^R}{\det(\Lambda)}\ll \frac{1}{\det(\Lambda)}T^{R-1}Y.
    \end{equation*}
    Thus we can complete the sum to get
    \begin{equation*}
        \sum_{l\leq T/Y}\mu(l)\#\big(\Lambda\cap\Ball_N(T/l)\big)=\frac{V_R}{\zeta(R)}\frac{T^R}{\det(\Lambda)}\Big(1+O\Big(\frac{Y\log T}{T}\Big)\Big).
    \end{equation*}
    For the second sum, we use Lemma \ref{Lemma:Degenerate point count} with $R_0=R-1$ so that we have
    \begin{equation*}
        \sum_{T/Y<l\leq T}\mu(l)\#\big(\Lambda\cap\Ball_N(T/l)\big)\ll \sum_{T/Y<l\leq T}\left(\frac{TY^{R-1}}{l\det(\Lambda)}+1\right).
    \end{equation*}
    After executing the sum over $l$, this is bounded by
    \begin{equation*}
        \frac{TY^{R-1}(\log Y)}{\det(\Lambda)}+T.
    \end{equation*}
    Now the lemma follows by using that $T\geq Y$.
\end{proof}
We will repeat the computation above in more general settings in Section \ref{subsection:second moment bounds} to pass from counting primitive integer points to all integer points. To apply any of the lemmas above, it may be necessary to give a bound on the size of the largest successive minimum of a given lattice. In our case a crude estimate will suffice.
\begin{lemma}\label{lemma:succ minima upper bound}
    Let $N\geq 1$ be an integer and $R\in\{1,\dots,N\}$. If $\Lambda\subset\RR^N$ is an integral lattice of rank $R$, there exists a $c>0$, depending at most on $R$, such that
    \begin{equation*}
        \lambda_R(\Lambda)\leq c\det(\Lambda).
    \end{equation*}
\end{lemma}
\begin{proof}
    Consider $\Lambda$ as a lattice of full rank in the vector space $\Span_{\RR}(\Lambda)$. Since $\Lambda$ is an integral lattice it is $\lambda_i(\Lambda)\geq 1$ for all $1\leq i\leq R$ and by Minkowski's second theorem there is a $\Tilde{c}>0$ such that we have
    \begin{equation*}
        \lambda_1(\Lambda)\cdots\lambda_R(\Lambda)\leq \Tilde{c}\det(\Lambda),
    \end{equation*}
    where $\Tilde{c}$ depends at most on $R$. Dividing by $\lambda_1(\Lambda)\cdots\lambda_{R-1}(\Lambda)$ completes the proof.
\end{proof}


\subsection{Determinants of certain lattices}$~$
In this subsection we briefly cover some formulas for determinants of a few lattices that will appear in Section \ref{section:proof_of_variance_bound}. The structure of this subsection is basically the same as \cite[Section 3.2]{BrLBSa} with some adjustments and simplifications that work in our setting. We start by recalling \cite[Definition 3.8]{BrLBSa}.

\begin{defn}
    Let $N\geq 1$ be an integer and $k\in\{1,\dots,N\}$. Given linearly independent vectors $\bc_1,\dots,\bc_k\in\ZZ^N$, we let $\Gcal(\bc_1,\dots,\bc_k)$ denote the greatest common divisor of the $k\times k$ minors of the $N\times k$ matrix whose columns are the vectors $\bc_1,\dots,\bc_k$.
\end{defn}
Next, recall the definitions (\ref{defn:Lambda_c lattice}) and (\ref{defn:lambda_c(Q) lattice}) of $\Lambda_{\bc}$ and $\Lambda_{\bc}^{(Q)}$ for $\bc\in\RR^N$ and $Q\geq 1$ an integer. Corresponding to \cite[Lemma 4]{LBpaper} and \cite[Lemma 3.10]{BrLBSa}, we have the following.
\begin{lemma}\label{lemma:determinant of Lambda}
    Let $N\geq 2$ be an integer and $k\in\{1,\dots,N-1\}$. Let $\bc_1,\dots,\bc_k\in\ZZ^N$ be linearly independent. Then the lattice $\Lambda_{\bc_1}\cap\cdots\cap\Lambda_{\bc_k}$ is primitive of rank $N-k$ and we have
    \begin{equation*}
        \det\big(\Lambda_{\bc_1}\cap\dots\cap\Lambda_{\bc_k}\big)=\frac{\det\big(\ZZ\bc_1\oplus\cdots\oplus\ZZ\bc_k\big)}{\Gcal(\bc_1,\dots,\bc_k)}.
    \end{equation*}
\end{lemma}
\begin{lemma}\label{lemma:determinants of Lambda intersections}
    Let $N\geq 2$ and $Q\geq 1$ be integers. Let $\bc,\bd\in\ZZ^N_{\prim}$ be linearly independent. We have
    \begin{equation*}
    \det\big(\Lambda_{\bc}^{(Q)}\cap\Lambda_{\bd}^{(Q)}\big)=\frac{Q^2}{\gcd\big(\Gcal(\bc,\bd),Q\big)}
    \end{equation*}
    and
    \begin{equation*}
        \det\big(\Lambda_{\bc}\cap\Lambda_{\bd}^{(Q)}\big)=\|\bc\|\frac{Q}{\gcd\big(\Gcal(\bc,\bd),Q\big)}.
    \end{equation*}
\end{lemma}
Next we introduce parts of \cite[Definition 3.13]{BrLBSa}.
\begin{defn}\label{def:dfrak}
    Let $N\geq 1$ be an integer and $R\in\{2,\dots,N\}$. Let $\bx,\by$ be two linearly independent vectors in $\ZZ^N$. We define $\dfrak_R(\bx,\by)$ to be the minimum determinant of a rank $R$ sublattice of $\ZZ^N$ containing $\bx$ and $\by$.
\end{defn}
By \cite[Lemma 3.14]{BrLBSa} we have the following Lemma.
\begin{lemma}\label{lemma:dfrak formel}
    Let $N\geq 2$ be an integer and $\bx,\by\in\ZZ^N$ be linearly independent. We have
    \begin{equation*}
        \dfrak_2(\bx,\by)=\frac{\det\big(\ZZ\bx\oplus\ZZ\by\big)}{\Gcal(\bx,\by)}.
    \end{equation*}
\end{lemma}
In particular we have that
\begin{equation}\label{Equation:dfrak=det(Gamma)}
    \dfrak_2(\bx,\by)=\det(\Lambda_{\bx}\cap\Lambda_{\by})
\end{equation}

\subsection{The average size of some determinants}$~$
We close out this section by giving bounds on the number of lattices with determinants of a certain size. For $X,Y,\Delta\geq 2$ we let
\begin{equation}\label{def:l(X,Y,Delta)}
    l(X,Y,\Delta)=\#\Bigg\{(\bx,\by)\in\Pri\times\Pri\,:\begin{array}{l}
         \dim(\text{Span}_{\RR}(\{\bx,\by\}))=2  \\
         \|\bx\|\leq X,\,\|\by\|\leq Y\\
         \dfrak_2(\bx,\by)\leq \Delta
    \end{array}
    \Bigg\}
\end{equation}
Note that this definition differs from \cite[Definition (3.24)]{BrLBSa} not only by fixing $r$ and $n$, but also by only considering vectors $\bx,\,\by$ that consist of primes. Also, by (\ref{Equation:dfrak=det(Gamma)}), this function actually counts pairs $\bx,\by$ in some box, such that $\det(\Lambda_{\bx}\cap\Lambda_{\by})$ is small. This will be of use in Section \ref{section:proof_of_variance_bound}. The main result of this subsection is a bound on $l(X,Y,\Delta)$ in the form of \cite[Lemma 3.21]{BrLBSa}. However in our setting that result does not suffice. Fortunately there is work of Holdridge available, dealing with cases including our case. The following lemma is a special case of \cite[Lemma 3.5]{Hol}.
\begin{lemma}\label{lemma:Holdridge l(X,Y) bound}
    Suppose $\eta>0$ and $3\leq Y^{\eta}\leq X\leq Y$. Then for all $\epsilon>0$ we have 
    \begin{equation*}
        l(X,Y,\Delta)\ll_{\eta}\frac{(XY)^2}{(\log X)^4(\log Y)^4}\Delta^2(\max\{1,\log\log\log Y\})^8+(XY)^2\Delta^{3/2}.
    \end{equation*}
\end{lemma}
This follows directly from choosing $n=3$, $r=2$ in Lemma 3.5 of \cite{Hol}.


\section{The variance upper bound}\label{section:proof_of_variance_bound}$~$

In this section we establish Proposition \ref{prop:variance_upper_bound}. Following \cite[Section 4]{BrLBSa}, we use Section \ref{subsection:volume estimate} to state volume estimates and estimates on certain sums over inverses of lattice determinants. In Section \ref{subsection:second moment bounds}, we combine the results of Sections \ref{section:geometry_of_numbers} and \ref{subsection:volume estimate} to obtain asymptotic formulas for second moments of the counting function $N_{\bfa}(B)$ and the local counting function $N_{\bfa}^{\loc}(B)$. Lastly we give bounds on certain first moments in Section \ref{section:first moment bounds} to finally obtain Proposition \ref{prop:variance_upper_bound}.

\subsection{Volume estimates and sums of determinants}\label{subsection:volume estimate}$~$
We start this section by reviewing some more definitions and lemmas of \cite{BrLBSa} (\cite[(4.1),(4,2),Lemma 4.2, Lemma 4.3]{BrLBSa}).
Let $N\geq 1$ be an integer and recall the definition (\ref{def:V_N}) of $V_N$. For $N\geq 2$ and $\bw,\bz\in\RR^N$, let
\begin{equation}\label{def:Ical volume}
    \Ical(\bw,\bz)=\vol\big(\big\{\bolt\in(\RR\bw)^{\perp}:|\langle\bz,\bolt\rangle|\leq\|\bolt\|\leq 1\big\}\big),
\end{equation}
and
\begin{equation*}
    \delta_{\bw,\bz}=\|\bw\|^2\|\bz\|^2-\langle\bw,\bz\rangle^2.
\end{equation*}
\begin{lemma}\label{lemma:Ical volume formula}
    Let $N\geq 3$ and let $\bw,\bz\in\RR^N$ be linearly independent. Then it is
    \begin{equation*}
        \Ical(\bw,\bz)=2\frac{N-2}{N-1}V_{N-2}\frac{\|\bw\|}{\delta_{\bw,\bz}^{1/2}}\Bigg(1+O\Big(\min\Big\{1,\frac{\|\bw\|^2}{\delta_{\bw,\bz}}\Big\}\Big)\Bigg),
    \end{equation*}
    where the implied constant depends at most on $N$.
\end{lemma}
Similarly, for $N\geq 2$ and $\bw,\bz\in\RR^N$, we let
\begin{equation}\label{def:Jcal volume}
    \Jcal(\bw,\bz)=\vol\big(\big\{\bolt\in\RR^N:|\langle\bw,\bolt\rangle|,|\langle,\bz,\bolt\rangle|\leq \|\bolt\|\leq 1\big\}\big).
\end{equation}
\begin{lemma}\label{lemma:Jcal volume formula}
    Let $N\geq 3$ and let $\bw,\bz\in\RR^N$ be linearly independent. Then we have that
    \begin{equation*}
        \Jcal(\bw,\bz)=4\frac{N-2}{N}V_{N-2}\frac{1}{\delta_{\bw,\bz}^{1/2}}\Bigg(1+O\Big(\min\big\{1,\frac{(\|\bw\|+\|\bz\|)^2}{\delta_{\bw,\bz}}\Big\}\Big)\Bigg).
    \end{equation*}
\end{lemma}
\vspace{0.5cm}
Proofs of these lemmas can be found in the reference cited above. The rest of this section is devoted to estimating  sums over inverses of lattice determinants. For this, let $\bx,\by\in\ZZ^4_{\prim}$ be two linearly independent vectors. By Lemma \ref{lemma:determinant of Lambda}, the lattice $\Lambda_{\bx}\cap\Lambda_{\by}$ is primitive of rank 2. This lattice will be important in Section \ref{subsection:second moment bounds} and we need to understand the size of its determinant at least on average. As a direct consequence of Lemma \ref{lemma:determinant of Lambda}, we have the following upper bound for its determinant.
\begin{cor}\label{corollary:Gamma_xy det bound}
    Let $N\geq 3$ be an integer and $\bx,\by\in\ZZ^N_{\prim}$ linearly independent. Then we have that
    \begin{equation*}
    \det\big(\Lambda_{\bx}\cap\Lambda_{\by}\big)\leq\|\bx\|\|\by\|.
    \end{equation*}
\end{cor}
\begin{proof}
    By Lemma \ref{lemma:determinant of Lambda}, it is
    \begin{equation*}
        \det\big(\Lambda_{\bx}\cap\Lambda_{\by}\big)\leq \det\big(\ZZ\bx\oplus\ZZ\by\big).
    \end{equation*}
    Since the vectors $\bx$ and $\by$ form a basis of the lattice $\ZZ\bx\oplus\ZZ\by$, we can compute its determinant (\ref{def:det of a lattice}) to be
    \begin{equation*}
        \det\big(\ZZ\bx\oplus\ZZ\by\big)^2=\det\Bigg(
        \begin{pmatrix}
            \|\bx\|^2 & \langle\bx,\by\rangle \\
            \langle\bx,\by\rangle & \|\by\|^2
        \end{pmatrix}
    \Bigg) = \|\bx\|^2\|\by\|^2-\langle\bx,\by\rangle^2\leq \|\bx\|^2\|\by\|^2.
    \end{equation*}
\end{proof}
The lattice $\Lambda_{\bx}\cap\Lambda_{\by}$ is only of rank $2$ when $\bx$ and $\by$ are linearly independent, else it is of rank $\geq 3$. In particular, the determinant $\det(\Lambda_{\bx}\cap\Lambda_{\by})$ is $0$ if $\bx$ and $\by$ are linearly dependent as can be seen in the proof above. Let $B\geq 2$ and recall the definition (\ref{defn:P(B)}) of $\Pri(B)$. We set
\begin{equation}\label{def:Omega(B)}
    \Omega(B)=\Big\{ (\bx,\by)\in \Pri(B)\times\Pri(B): \bx\neq\by\Big\}.
\end{equation}
If $(\bx,\by)\in\Omega(B)$, then $\bx$ and $\by$ are linearly independent since their entries consist of primes. Hence we can let
\begin{equation}\label{defn:E(B)}
    E(B)=(\log B)^8\sum_{(\bx,\by)\in\Omega(B)}\frac{1}{\det(\Lambda_{\bx}\cap\Lambda_{\by})}.
\end{equation}
We proceed to give a lower and an upper bound for $E(B)$.
\begin{lemma}\label{lemma:E(B) upper and lower bounds}
    For $B\geq 3$ we have
    \begin{equation*}
        B^6\ll E(B)\ll B^6\max\{1,\log\log\log B\}^8.
    \end{equation*}
\end{lemma}
\begin{proof}
    We start with the lower bound. By the previous corollary it is
    \begin{equation*}
        E(B)\geq(\log B)^8\sum_{(\bx,\by)\in\Omega(B)}\frac{1}{\|\bx\|\|\by\|}\geq \frac{(\log B)^8}{B^2}\#\Omega(B).
    \end{equation*}
    Note that $\Omega(B)$ consists of all tuples of vectors of primes up to size $B$ where only diagonal is taken out. Thus, if we let $\pi(X)$ be the number of primes up to $X$, it is
    \begin{equation*}
        \#\Omega(B)\geq \pi(B)^8-\pi(B)^4\gg \frac{B^8}{(\log B)^8}
    \end{equation*}
    for $B$ large enough. This confirms the lower bound. For the upper bound, we make use of Lemma \ref{lemma:Holdridge l(X,Y) bound} by cutting the sum into dyadic pieces. Recall the definition (\ref{def:l(X,Y,Delta)}) of $l(X,Y,\Delta)$. We have
    \begin{align*}
        E(B)\ll (\log B)^8\sum_{X\leq Y\ll B}\sum_{\Delta\ll XY}\frac{1}{\Delta}l(X,Y,\Delta).
    \end{align*}
    Fix a $0<\eta<1$. Then, by Lemma \ref{lemma:Holdridge l(X,Y) bound}, for $Y^{\eta}\leq X\leq Y$ large enough, we have that $l(X,Y,\Delta)$ is bounded by $(XY\Delta)^2(\log X)^{-4}(\log Y)^{-4}(\log\log\log Y)^8$. By using \cite[Lemma 3.21]{BrLBSa} (with $n=3$ and $r=2$) in the remaining cases, we get
    \begin{align*}
        \frac{E(B)}{(\log B)^8}&\ll\sum_{Y^{\eta}\leq X\leq Y\ll B}\sum_{\Delta\ll XY}\frac{1}{\Delta}l(X,Y,\Delta)
        +\sum_{X< Y^{\eta}\ll B}\sum_{\Delta\ll XY}\frac{1}{\Delta}l(X,Y,\Delta).\\
        &\ll\sum_{Y^{\eta}\leq X\leq Y\ll B}(XY)^3(\log X)^{-4}(\log Y)^{-4}(\log\log\log Y)
        +\sum_{X<Y^{\eta}\ll B}(XY)^3\\
        &\ll B^6(\log B)^{-8}(\log\log\log B)^8+B^{3(1+\eta)}.
    \end{align*}
    This completes the proof.
\end{proof}
As done in \cite{BrLBSa}, we next give a bound on a certain weighted average over the inverses of lattice determinants. Let $\bx,\by\in\ZZ^4$ linearly independent and set
\begin{equation}\label{def:Delta(x,y)}
    \Delta(\bx,\by)=\frac{\|\bx\|\|\by\|}{\det(\ZZ\bx\oplus\ZZ\by)}.
\end{equation}
Now recall the definitions (\ref{defn:alpha}), (\ref{def:w}) and (\ref{def:W}) of $\alpha, w$ and $W$ and let
\begin{equation*}
    \text{rad}(W)=\prod_{p\leq w}p
\end{equation*}
the radical of $W$. Let
\begin{equation}\label{def:Ecal_x,y}
    \Ecal_{\bx,\by}(B)=\min\Big\{1,\frac{\Delta(\bx,\by)^2}{\alpha^2}\Big\}+\textbf{1}_{\Gcal(\bx,\by)\nmid W/\text{rad}(W)}.
\end{equation}
The next lemma is then concerned with a saving on 
\begin{equation}\label{def:F(B)}
F(B)=(\log B)^8\sum_{(\bx,\by)\in\Omega(B)}\frac{\Ecal_{\bx,\by}(B)}{\det(\Lambda_{\bx}\cap\Lambda_{\by})}
\end{equation}
over the trivial bound from the previous lemma.
\begin{lemma}\label{lemma:F(B) saving}
    Let $B\geq 2$ and $\epsilon>0$. We have
    \begin{equation*}
        F(B)\ll_{\epsilon}\frac{B^6}{(\log\log B)^{1/2-\epsilon}}.
    \end{equation*}
\end{lemma}
\begin{proof}
    We can mostly follow the proof of \cite[Lemma 4.6]{BrLBSa} but we have to use Lemma \ref{lemma:Holdridge l(X,Y) bound} in our setting. We split $F(B)$ into two parts, corresponding to the summands of $\Ecal_{\bx,\by}(B)$. Namely we let
    \begin{equation*}
         F_1(B):=\sum_{(\bx,\by)\in\Omega(B)}\frac{1}{\det(\Gamma_{\bx,\by})}\min\Big\{1,\frac{\Delta(\bx,\by)^2}{\alpha^2}\Big\},
    \end{equation*}
    and
    \begin{equation*}
        F_2(B):=\sum_{(\bx,\by)\in\Omega(B)}\frac{1}{\det(\Gamma_{\bx,\by})}\textbf{1}_{\Gcal(\bx,\by)\nmid W/\text{rad}(W)}.
    \end{equation*}
    Then we can recover $F(B)$ as $F(B)=(\log B)^8(F_1(B)+F_2(B))$. We deal with $F_1(B)$ first. Recall the definition \ref{def:dfrak} of $\dfrak_2(\bx,\by)$. By Lemma \ref{lemma:dfrak formel} we have that
    \begin{equation*}
        \Delta(\bx,\by)\leq\frac{\|\bx\|\|\by\|}{\dfrak_2(\bx,\by)}.
    \end{equation*}
    After breaking the sizes of $\|\bx\|,\|\by\|$ and $\dfrak_2(\bx,\by)$ into dyadic intervals, we obtain
    \begin{equation*}
        F_1(B)\ll\sum_{X\leq Y\ll B}\sum_{\Delta_2\ll XY}\frac{1}{\Delta_2}\Big(\min \Big\{1,\frac{(XY)^2}{\Delta_2^2\alpha^2}\Big\}\Big)l(X,Y,\Delta_2).
    \end{equation*}
    We can replace the minimum with the estimate
    \begin{equation*}
        \min\Big\{1,\frac{(XY)^2}{\Delta_2^2\alpha^2}\Big\}\leq \frac{(XY)^{1/2}}{\Delta_2^{1/2}\alpha^{1/2}}.
    \end{equation*}
    Next fix a $0<\eta<1$. For $Y^{\eta}\leq X\leq Y$  large enough, we may use Lemma \ref{lemma:Holdridge l(X,Y) bound}. In the case where $X<Y^{\eta}$ we again use \cite[Lemma 3.21]{BrLBSa}. This yields
    \begin{align*}
        F_1(B)&\ll\frac{1}{\alpha^{1/2}}\sum_{Y^{\eta}\leq X\leq Y\ll B}\frac{(XY)^3(\log\log\log Y)^8}{(\log X)^4(\log Y)^4}+\frac{1}{\alpha^{1/2}}\sum_{X<Y^{\eta}\leq Y\ll B}(XY)^3\\
        &\ll \frac{1}{\alpha^{1/2}}\frac{B^6}{(\log B)^8}(\log\log\log B)^8+B^{3(1+\eta)}.
    \end{align*}
    Next we deal with $F_2(B)$. For a prime $p$ and an integer $m$, let $v_p(m)$ be the p-adic valuation of $m$. If $m$ does not divide $W/\text{rad}(W)$, then there is either a prime $p>w$ dividing $m$ or there is a prime $p\leq w$ such that
    \begin{equation*}
        v_p(m)>v_p\Big(\frac{W}{\text{rad}(W)}\Big)=\Big\lceil\frac{\log w}{\log p}\Big\rceil\geq\frac{\log w}{\log p}.
    \end{equation*}
    Hence, in either case it is $m>w$. Therefore, if $\bx,\by$ are linearly independent and $\Gcal(\bx,\by)$ does not divide $W/\text{rad}(W)$, we must have that $\Gcal(\bx,\by)>w$ and thus by Lemma \ref{lemma:dfrak formel} we find that
    \begin{equation*}
        \dfrak_2(\bx,\by)\leq\frac{\|\bx\|\|\by\|}{w}.
    \end{equation*}
    Again splitting the sum in $F_2(B)$ into dyadic pieces, we obtain
    \begin{equation*}
        F_2(B)\ll\sum_{X\leq Y\ll B}\sum_{\Delta_2\ll XY/w}\frac{1}{\Delta_2}l(X,Y,\Delta_2).
    \end{equation*}
    We deal with this in the same way as we did with $F_1(B)$. Again fix a $0<\eta<1$. Then by Lemma \ref{lemma:Holdridge l(X,Y) bound} and \cite[Lemma 3.21]{BrLBSa}, we have that
    \begin{align*}
        F_2(B)&\ll \frac{1}{w}\sum_{Y^{\eta}\leq X\leq Y\ll B}\frac{(XY)^3(\log\log\log Y)^8}{(\log X)^4(\log Y)^4}+\frac{1}{w}\sum_{X<Y^{\eta}\leq Y\ll B}(XY)^3\\
        &\ll\frac{1}{w}\frac{B^6}{(\log B)^8}(\log\log\log B)^8+B^{3(1+\eta)}.
    \end{align*}
    Recalling the definition (\ref{def:w}) of $w$ once again, we see that the saving is large enough.
\end{proof}

\subsection{The second moment bounds}\label{subsection:second moment bounds}$~$
In this section we prove asymptotic formulas for the second moments of $N_{\bfa}(B)$, $N_{\bfa}^{\loc}(B)$ and a mixed term. We follow the \cite[Section 4.5]{BrLBSa} and introduce some similar notation. Recall the definition (\ref{defn:N_a(B)}) of $N_{\bfa}(B)$. We let 
\begin{equation}\label{defn:D(A,B)}
    D(A,B)=\sum_{\substack{\bfa\in\ZZ^4_{\prim}\\\|\bfa\|\leq A}}N_{\bfa}(B)^2-(\log B)^4\sum_{\substack{\bfa\in\ZZ^4_{\prim}\\\|\bfa\|\leq A}}N_{\bfa}(B).
\end{equation}
This corresponds to the second moment of $N_{\bfa}(B)$ with its diagonal removed. That is, we interpret the counting function $N_{\bfa}(B)^2$ as counting tuples $(\bx,\by)$ in $\Omega(B)$ satisfying $\langle\bx,\bfa\rangle=0$ and $\langle\by,\bfa\rangle=0$. Now recall the definition (\ref{defn:N^loc}) of $N_{\bfa}^{\loc}(B)$. To define the second moment for that counting function and a mixed moment containing both $N_{\bfa}(B)$ and $N_{\bfa}^{\loc}(B)$, we first need to define the terms removing their diagonal contributions. For this, recall the definitions (\ref{defn:lambda_c(Q) lattice}) and (\ref{defn:C^gamma}) of $\Lambda_{\bx}^{(Q)}$ and $\Ccal_{\bx}^{(\gamma)}$. Let
\begin{equation}\label{def:Delta^mix and ^loc}
    \Delta^{\mix}_{\bfa}(B)=(\log B)^8\frac{\alpha W}{\|\bfa\|}\sum_{\substack{\bx\in\Pri(B)\\ \bfa\in\Lambda_{\bx}}}\frac{1}{\|\bx\|},\quad\text{and}\quad \Delta^{\loc}_{\bfa}(B)=(\log B)^8\frac{\alpha^2W^2}{\|\bfa\|^2}\sum_{\substack{\bx\in\Pri(B)\\\bfa\in\Lambda_{\bx}^{(W)}\cap\Ccal_{\bx}^{(\alpha)}}}\frac{1}{\|\bx\|^2}.
\end{equation}
We can now define the mixed second moment as
\begin{equation}\label{defn:D^mix}
    D^{\mix}(A,B)=\sum_{\substack{\bfa\in\ZZ^4_{\prim}\\\|\bfa\|\leq A}}N_{\bfa}(B)N_{\bfa}^{\loc}(B)-\sum_{\substack{\bfa\in\ZZ^4_{\prim}\\\|\bfa\|\leq A}} \Delta^{\mix}_{\bfa}(B),
\end{equation}
and the local second moment as
\begin{equation}\label{defn:D^loc}
    D^{\loc}(A,B)=\sum_{\substack{\bfa\in\ZZ^4_{\prim}\\\|\bfa\|\leq A}}N_{\bfa}^{\loc}(B)^2-\sum_{\substack{\bfa\in\ZZ^4_{\prim}\\\|\bfa\|\leq A}}\Delta^{\loc}_{\bfa}(B).
\end{equation}
As is done in \cite[(4.24), (4,25)]{BrLBSa}, for a lattice $\Lambda\subset\ZZ^4$, a bounded region $\Rcal\subset\RR^4$ and any integer $k\geq 0$ we let
\begin{equation}\label{defn:Scal_k and ^star}
    \Scal_k(\Lambda;\Rcal)=\sum_{\bfa\in(\Lambda\setminus\{0\})\cap\Rcal}\frac{1}{\|\bfa\|^k},\quad\text{and}\quad \Scal_k^*(\Lambda;\Rcal)= \sum_{\bfa\in\Lambda\cap\ZZ^4_{\prim}\cap\Rcal}\frac{1}{\|\bfa\|^k}.
\end{equation}
We are now set up to prove the main lemmas of this section. Recall the definition (\ref{defn:E(B)}) of $E(B)$.
\begin{lemma}\label{lemma:D(A,B) asymptotic formula}
    Let $C>0$. For $A\geq B^2(\log B)^{C+1}$ we have
    \begin{equation*}
        D(A,B)=\frac{\pi}{\zeta(2)}A^2E(B)\Bigg(1+O\Big(\frac{1}{(\log B)^C}\Big)\Bigg).
    \end{equation*}
\end{lemma}
\begin{proof}
    For a tuple $(\bx,\by)$ in $\Omega(B)$, we let
    \begin{equation*}
        \Gamma_{\bx,\by}=\Lambda_{\bx}\cap\Lambda_{\by}.
    \end{equation*}
    Since the diagonal contribution got removed in the definition (\ref{defn:D(A,B)}) of $D(A,B)$, we can write
    \begin{align*}
        D(A,B)&=(\log B)^8\sum_{\substack{\bx,\by\in\Pri(B)\\\bx\neq\pm\by}} \#\{\bfa\in\ZZ^4_{\prim}:\|\bfa\|\leq A,\,\langle\bfa,\bx\rangle=\langle\bfa,\bx\rangle=0\}\\
        &=(\log B)^8\sum_{(\bx,\by)\in\Omega(B)}\Scal_0^*(\Gamma_{\bx,\by};\Ball_4(A)).
    \end{align*}
    We can deal with the inner counting function by applying Lemma \ref{lemma:primitive lattice points}. For this to succeed, we check for the size of the successive minima of $\Gamma_{\bx,\by}$. Since the tuple $(\bx,\by)$ is from $\Omega(B)$, the lattice in question is primitive of rank $2$ and by Corollary \ref{corollary:Gamma_xy det bound} its determinant is bounded by $B^2$. Thus, by Lemma \ref{lemma:succ minima upper bound}, there exists a $c>0$ such that all of its successive minima are bounded by $cB^2$. In our case, if $C>0$ and $B$ is large enough, then $A\geq B^2(\log B)^{C+1}$ is large enough to apply the lemma. Hence we get that
    \begin{equation*}
        \Scal_0^*(\Gamma_{\bx,\by};\Ball_4(A)))=\frac{\pi}{\zeta(2)}\frac{A^2}{\det(\Gamma_{\bx,\by})}\Big(1+O\Big(\frac{B^2\log B}{A}\Big)\Big)+ O\big(A\log B\big).
    \end{equation*}
    Again, since $\det(\Gamma_{\bx,\by})\leq B^2$, we also have that
    \begin{equation*}
        A\log B\leq\frac{A^2}{\det(\Gamma_{\bx,\by})}\frac{B^2\log B}{A},
    \end{equation*}
    so that we can ignore the last error term. Putting this back into the sum over $(\bx,\by)\in\Omega(B)$ completes the proof.
\end{proof}
We now move on to the mixed term.

\begin{lemma}\label{lemma:D^mix asymptotic formula}
    Let $A\geq B^2(\log B)^2$. For all $\epsilon>0$ we have that
    \begin{equation*}
        D^{\mix}(A,B)=\frac{\pi}{\zeta(2)}A^2E(B)\Big(1+O\Big(\frac{1}{(\log\log B)^{1/2-\epsilon}}\Big).
    \end{equation*}
\end{lemma}
\begin{proof}
    For $(\bx,\by)\in\Omega(B)$, recall the definitions (\ref{defn:lambda_c(Q) lattice}) and (\ref{defn:C^gamma}) of $\Lambda_{\bx}^{(W)}$ and the region $\Ccal_{\by}^{(\alpha)}$. We let
    \begin{equation*}
        \Gamma_{\bx,\by}^{\mix}(W)=\Lambda_{\bx}\cap\Lambda_{\by}^{(W)},
    \end{equation*}
    and
    \begin{equation*}
        \Tcal^{\mix}_{\by}(A,\alpha)=\Ball_4(A)\cap\Ccal_{\by}^{(\alpha)}.
    \end{equation*}
    In the same way as in the lemma before, we can rewrite $D^{\mix}(A,B)$ into
    \begin{equation*}
        D^{\mix}(A,B)=(\log B)^8\alpha W\sum_{(\bx,\by)\in\Omega(B)}\frac{\Scal_1^*\big(\Gamma_{\bx,\by}^{\mix}(W);\Tcal_{\by}^{\mix}(A,\alpha)\big)}{\|\bx\|}.
    \end{equation*}
    Following the proof of \cite[Lemma 4.10]{BrLBSa}, we handle $\Scal_1^*\big(\Gamma_{\bx,\by}^{\mix}(W);\Tcal_{\by}^{\mix}(A,\alpha)\big)$ by applying a Möbius inversion to get rid of the coprimality conditions in the counting function. After that we can remove the weight $\|\bx\|^{-1}$ via summation by parts to then use Lemma \ref{Lemma:General point count} and obtain an asymptotic formula for $D^{\mix}(A,B)$. By Möbius inversion we have
    \begin{equation*}
        \mathcal{S}_1^*\big(\Gamma^{\mix}_{\bx,\by}(W),\mathcal{T}_{\by}^{\mix}(A,\alpha)\big)=\sum_{l\leq A}\frac{\mu(l)}{l}\mathcal{S}_1\Big(\Gamma^{\mix}_{\bx,\by}\Big(\frac{W}{\gcd(l,W)}\Big),\mathcal{T}_{\by}^{\mix}(\frac{A}{l},\alpha)\Big).
    \end{equation*}
    Note that by Lemma \ref{lemma:determinants of Lambda intersections}, it is
    \begin{equation}\label{equation:Determinant of Gamma^mix}
        \det\Big(\Gamma_{\bx,\by}^{\mix}\Big(\frac{W}{\gcd(W,l)}\Big)\Big)=\|\bx\|\frac{W}{\gcd(W,l)\gcd(\Gcal(\bx,\by),\frac{W}{\gcd(W,l)})}\leq \|\bx\|W.
    \end{equation}
    Like in the proof of Lemma \ref{lemma:primitive lattice points}, we split the sum over $l$ into 2 parts according to the size of $A/l$. First, for $u\geq 1$ real it is
    \begin{equation}\label{equation:S_1 bound against larger lattice}
        \mathcal{S}_1\Big(\Gamma^{\mix}_{\bx,\by}\Big(\frac{W}{\gcd(l,W)}\Big),\mathcal{T}_{\by}^{\mix}(u,\alpha)\Big)\leq \mathcal{S}_1\big(\Lambda_{\bx},\Ball_4(u)\big)\ll \sum_{U\ll u}\frac{1}{U}\mathcal{S}_0\big(\Lambda_{\bx},\Ball_4(u)\big).
    \end{equation}
    Let $\lambda(\bx)$ be the largest successive minimum of $\Lambda_{\bx}$. When $U\geq\lambda(\bx)$, we can apply Lemma \ref{Lemma:General point count}. Else it is $U<\lambda(\bx)$. In that case we may use Lemma \ref{Lemma:Degenerate point count} with $M=1/2$ and $R_0=2$. Recall from Lemma \ref{lemma:determinant of Lambda} that we have $\det(\Lambda_{\bx})=\|\bx\|$, so that we obtain
    \begin{equation*}
        \Scal_0\big(\Lambda_{\bx},\Ball_4(U)\big)\ll \frac{U^3}{\|\bx\|}+\frac{U\lambda(\bx)^2}{\|\bx\|}+1.
    \end{equation*}
    Now using that $\lambda(\bx)\ll\det(\Lambda_{\bx})=\|\bx\|$, this shows that for any $u\geq 1$ we have
    \begin{equation}\label{equation:S_1 bound im mixed lemma}
        \mathcal{S}_1\Big(\Gamma^{\mix}_{\bx,\by}\Big(\frac{W}{\gcd(l,W)}\Big),\mathcal{T}_{\by}^{\mix}(u,\alpha)\Big)\ll
        \frac{u^2}{\|\bx\|}+\|\bx\|\log u+1.
    \end{equation}
Let $c>0$ be the constant given in Lemma \ref{lemma:succ minima upper bound} such that the successive minima of $\Gamma^{\mix}_{\bx,\by}(W/\gcd(W,l))$ are bounded by $cW\|\bx\|$. For the range where $A/cW\|\bx\|<l\leq A$, we get that
    \begin{align*}
        \sum_{A/cW\|\bx\|<l\leq A}\frac{\mu(l)}{l}\mathcal{S}_1\Big(\Gamma^{\mix}_{\bx,\by}\Big(\frac{W}{\gcd(l,W)}\Big),\mathcal{T}_{\by}^{\mix}(u,\alpha)\Big) &\ll
        \sum_{A/cW\|\bx\|<l\leq A}\frac{1}{l}\Big(\frac{A^2}{l^2\|\bx\|}+\|\bx\|\log A+1\Big)\\
        &\ll \sum_{A/cW\|\bx\|<l\leq A}\frac{1}{l}\big(W^2\|\bx\|+\|\bx\|\log A+1\big)\\
        &\ll (W^2\|\bx\|+1)(\log A)^2
    \end{align*}
    For the remaining $l\leq A$ it is $A/l\geq cW\|\bx\|$ and we wish to apply Lemma \ref{Lemma:General point count}. To do this, let
    \begin{equation*}
        S_{\bx,\by}^{\mix}(A,B,l)=\Scal_1\Big(\Gamma_{\bx,\by}^{\mix}\Big(\frac{W}{\gcd(W,l)}\Big);\Tcal_{\by}^{\mix}\Big(\frac{A}{l},\alpha\Big)\setminus\Tcal_{\by}^{\mix}\big(cW\|\bx\|,\alpha\big)\Big).
    \end{equation*}
    With the region $\Tcal_{\by}^{\mix}\big(cW\|\bx\|,\alpha\big)$ removed from the count, we can apply partial summation to get rid of the weight and apply the lemma. Also by another application of (\ref{equation:S_1 bound im mixed lemma}), we have that
    \begin{equation}\label{equation:S_1^* formula for mixed}
        \mathcal{S}_1^*(\Gamma^{\mix}_{\bx,\by}(W),\mathcal{T}_{\by}^{\mix}(A,\alpha))=\sum_{l\leq A/cW\|\bx\|}\frac{\mu(l)}{l}S_{\bx,\by}^{\mix}(A,B,l) + O\Big(\Big(W\|\bx\|+1\Big)(\log A)^2\Big).
    \end{equation}
    Now we carry out the partial summation. It is
    \begin{align}\label{equation:S^mix_x,y formula}
        S_{\bx,\by}^{\mix}(A,B,l)=&\frac{l}{A}\Scal_0\Big(\Gamma_{\bx,\by}^{\mix}\Big(\frac{W}{\gcd(W,l)}\Big);\Tcal^{\mix}_{\by}\Big(\frac{A}{l},\alpha\Big)\Big) \nonumber\\
        &+ \int_{cW\|\bx\|}^{A/l}\Scal_0\Big(\Gamma_{\bx,\by}^{\mix}\Big(\frac{W}{\gcd(W,l)}\Big);\Tcal_{\by}^{\mix}(t,\alpha)\Big)\frac{\text{d}t}{t^2}
        + O\big(W\|\bx\|\big).
    \end{align}
    Recall that by our choice of $c$, all successive minima of $\Gamma_{\bx,\by}^{\mix}(W/\gcd(W,l))$ are bounded by $cW\|\bx\|$. Hence we can now apply Lemma \ref{Lemma:General point count} with $I=1$ and $\gamma=\alpha$. For this, let
    \begin{equation*}
        \mathfrak{I}_{\bx,\by}(\alpha)=\text{Span}_{\RR}(\Lambda_{\bx})\cap\Tcal_{\by}^{\mix}(1,\alpha),
    \end{equation*}
    so that we have
    \begin{align*}
        \mathcal{S}_0\Big(\Gamma^{\mix}_{\bx,\by}\Big(\frac{W}{\gcd(W,l)},\mathcal{T}^{\mix}_{\by}(t,\alpha)\Big)&= t^3\Big(\det\Big(\Gamma^{\mix}_{\bx,\by}\Big(\frac{W}{\gcd(W,l)}\Big)\Big)^{-1}\cdot\Big(\vol(\mathfrak{I}_{\bx,\by}   (\alpha))+O\Big(\frac{W\|\bx\|}{t}\Big)\Big).
    \end{align*}
    The determinant in the display above is given in (\ref{equation:Determinant of Gamma^mix}). Note that the sum in (\ref{equation:S_1^* formula for mixed}) is only over squarefree $l$, hence we can write
    \begin{equation*}
        \gcd\Big(\Gcal(\bx,\by),\frac{W}{\gcd(W,l)}\Big)=\Gcal(\bx,\by)\big(1+O(\textbf{1}_{\Gcal(\bx,\by)\nmid W/\text{rad}(W)})\big).
    \end{equation*}
    To deal with the volume, recall Definition (\ref{def:Ical volume}). Noting that
    \begin{equation*}
        \mathfrak{I}_{\bx,\by}(\alpha)=\Ical\Big(2\alpha\frac{\bx}{\|\bx\|},2\alpha\frac{\by}{\|\by\|}\Big),
    \end{equation*}
    we may apply Lemma \ref{lemma:Ical volume formula}. Recalling the definition (\ref{def:Delta(x,y)}) of $\Delta(\bx,\by)$, the lemma yields that
    \begin{equation*}
        \vol(\mathfrak{I}_{\bx,\by}(\alpha))=\frac{2}{3}V_2\frac{\Delta(\bx,\by)}{\alpha}\Big(1+O\Big(\min\Big\{1,\frac{\Delta(\bx,\by)^2}{\alpha^2}\Big\}\Big)\Big).
    \end{equation*}
    Now we are in position to carry out the integration in (\ref{equation:S^mix_x,y formula}). Recall the definition (\ref{def:Ecal_x,y}) of $\Ecal_{\bx,\by}(B)$. We get that
    \begin{equation*}
        S^{\mix}_{\bx,\by}(A,B,l)=V_2\frac{A^2}{\alpha W}\frac{\gcd(l,W)}{l^2}\frac{\Delta(\bx,\by)\Gcal(\bx,\by)}{\|\bx\|}\Big(1+O\Big(\Ecal_{\bx,\by}(B)+\frac{\alpha W\|\bx\|l}{A}\Big)\Big)
        +O\big(W\|\bx\|\big).
    \end{equation*} 
    We put this back into (\ref{equation:S_1^* formula for mixed}) and note that
    \begin{equation*}
        \frac{\Delta(\bx,\by)\Gcal(\bx,\by)}{\|\bx\|\|\by\|}=\frac{1}{\det(\Lambda_{\bx}\cap\Lambda_{\by})}.
    \end{equation*}
    To simplify further, consider
    \begin{equation*}
        \sum_{l\leq A/cW\|\bx\|}\mu(l)\frac{\gcd(W,l)}{l^3}=\sum_{l\leq w}\mu(l)\frac{\gcd(W,l)}{l^3}+O\Big(\frac{W\|\bx\|}{A}+\frac{1}{w}\Big).
    \end{equation*}
    For all $l\leq w$ we have that $\gcd(W,l)=l$, so that we can replace the sum on the right hand side above by
    \begin{equation*}
        \frac{1}{\zeta(2)}+O\Big(\frac{1}{w}\Big).
    \end{equation*}
    After recalling the definition (\ref{def:F(B)}) of $F(B)$ and using that $\alpha\ll\log B$ we obtain
    \begin{equation*}
        D^{\mix}(A,B)=\frac{\pi}{\zeta(2)}A^2E(B)\Big(1+O\Big(\frac{1}{w}+\frac{F(B)}{B^6}\Big)\Big)+O\Big((\log B)^9A(\log A)W\sum_{(\bx,\by)\in\Omega(B)}\frac{\|\bx\|}{\det(\Lambda_{\bx}\cap\Lambda_{\by})}\Big).
    \end{equation*}
    We can use the same argument as in the proof of Lemma \ref{lemma:E(B) upper and lower bounds} to see that the sum in the last error term is bounded by $B^7(\log B)^{-7}$. By (\ref{W_bound}), it is $W\ll (\log B)^3$ and since $B^2<A$, we get that the last error term in total is bounded by $A^2B^{5+\epsilon}$ for $\epsilon>0$. Now (\ref{def:w}) and Lemma \ref{lemma:F(B) saving} show that the remaining error terms are small enough to obtain the lemma.
\end{proof}

Next we deal with the local term.
\begin{lemma}\label{lemma:D^loc asymptotic formula}
    Let $B^2\leq A$. For all $\epsilon>0$ we have
    \begin{equation*}
        D^{\loc}(A,B)=\frac{\pi}{\zeta(2)}A^2E(B)\Big(1+O\Big(\frac{1}{(\log\log B)^{1/2-\epsilon}}\Big)\Big).
    \end{equation*}
\end{lemma}
\begin{proof}
    Again, we follow the proof of \cite[Lemma 4.11]{BrLBSa}. In this case we let
    \begin{equation*}
        \Gamma^{\loc}_{\bx,\by}(W)=\Lambda_{\bx}^{(W)}\cap\Lambda_{\by}^{(W)},
    \end{equation*}
    and
    \begin{equation*}
        \Tcal^{\loc}_{\bx,\by}(A,\alpha)=\Ball_4(A)\cap\Ccal_{\bx}^{(\alpha)}\cap\Ccal_{\by}^{(\alpha)}.
    \end{equation*}
    For any given lattice $\Lambda\subset\ZZ^4$, any bounded region $\Rcal\subset\RR^4$ and an integer $k\geq 0$ again recall the definition (\ref{defn:Scal_k and ^star}) of the sums $\Scal_k^*(\Lambda;\Rcal)$ and $\Scal_k(\Lambda;\Rcal)$. Then, analogous to the lemmas before, we have
    \begin{equation}\label{equation:D^loc representation as S_2^* sum}
        D^{\loc}(A,B)=(\log B)^8\frac{\alpha^2W^2}{2}\sum_{(\bx,\by)\in\Omega(B)} \frac{\Scal_2^*\big(\Gamma^{\loc}_{\bx,\by}(W);\Tcal^{\loc}_{\bx,\by}(A,\alpha)\big)}{\|\bx\|\|\by\|}.
    \end{equation}
    We apply a Möbius inversion to get
    \begin{equation}\label{equation:S_2 Moebius convolution representation}
        \Scal_2^*\big(\Gamma^{\loc}_{\bx,\by}(W);\Tcal^{\loc}_{\bx,\by}(A,\alpha)\big) = \sum_{l\leq A}\frac{\mu(l)}{l^2}\Scal_2\Big(\Gamma_{\bx,\by}^{\loc}\Big(\frac{W}{\gcd(W,l)}\Big);\Tcal_{\bx,\by}^{\loc}\Big(\frac{A}{l},\alpha\Big)\Big).
    \end{equation}
    Again, for large $l$ the contribution is small. Like in (\ref{equation:S_1 bound against larger lattice}), for $u\geq 1$ real we have
    \begin{equation*}
        \Scal_2\Big(\Gamma_{\bx,\by}^{\loc}\Big(\frac{W}{\gcd(W,l)}\Big);\Tcal_{\bx,\by}^{\loc}(u,\alpha)\Big)\leq \Scal_2(\ZZ^4;\Ball_4(u))\ll u^2.
    \end{equation*}
    So for $l>A/W$, the contribution to the sum is
    \begin{equation*}
        \sum_{l>A/W}\frac{\mu(l)}{l^2} \Scal_2\Big(\Gamma_{\bx,\by}^{\loc}\Big(\frac{W}{\gcd(W,l)}\Big);\Tcal_{\bx,\by}^{\loc}\Big(\frac{A}{l},\alpha\Big)\Big) \ll \frac{W^3}{A}.
    \end{equation*}
    To deal with the remaining summands in (\ref{equation:S_2 Moebius convolution representation}), we let 
    \begin{equation*}
        S^{\loc}_{\bx,\by}(A,B;l)=\Scal_2\Big(\Gamma^{\loc}_{\bx,\by}\Big(\frac{W}{\gcd(W,l)}\Big);\Tcal^{\loc}_{\bx,\by}\Big(\frac{A}{l},\alpha\Big)\setminus\Tcal^{\loc}_{\bx,\by}(W,\alpha)\Big),
    \end{equation*}
    so that we have
    \begin{equation}\label{equation:S_2^star as S^loc representation}
        \Scal_2^*(\Gamma^{\loc}_{\bx,\by}(W),\Tcal^{\loc}_{\bx,\by}(A,\alpha))=\sum_{l\leq A/W} \frac{\mu(l)}{l^2} S^{\loc}_{\bx,\by}(A,B;l) +O\Big(\frac{W^3}{A}\Big).
    \end{equation}
    Let $d$ be an integer dividing $W$. For $1\leq i\leq 4$, the vectors $W\textbf{e}_i$ lie in $\Gamma_{\bx,\by}^{\loc}(d)$. Therefore it is
    \begin{equation*}
        \lambda_4\Big(\Gamma_{\bx,\by}^{\loc}\Big(\frac{W}{\gcd(W,l)}\Big)\Big)\leq W.
    \end{equation*}
    This allows us to apply Lemma \ref{Lemma:General point count} to $S_{\bx,\by}^{\loc}(A,B;l)$ after removing the weight with partial summation. We have
    \begin{multline*}
        S^{\loc}_{\bx,\by}(A,B;l)=\frac{l^2}{A^2}\Scal_0\Big(\Gamma^{\loc}_{\bx,\by}\Big(\frac{W}{\gcd(W,l)}\Big);\Tcal^{\loc}_{\bx,\by}\Big(\frac{A}{l},\alpha\Big)\Big)\\
        +2\int_W^{A/l}\Scal_0\Big(\Gamma^{\loc}_{\bx,\by}\Big(\frac{W}{\gcd(W,l)}\Big);\Tcal^{\loc}_{\bx,\by}(t,\alpha)\Big)\frac{dt}{t^3}+O\Big(\frac{W^3}{A}\Big).
    \end{multline*}
    Now, for any $t\in [W,A/l]$, the application of Lemma \ref{Lemma:General point count} yields 
    \begin{equation*}
        \Scal_0\Big(\Gamma^{\loc}_{\bx,\by}\Big(\frac{W}{\gcd(W,l)}\Big);\Tcal^{\loc}_{\bx,\by}(t,\alpha)\Big)=\frac{t^2}{\det\Big(\Gamma^{\loc}_{\bx,\by}\Big(\frac{W}{\gcd(W,l)}\Big)\Big)}\Bigg(\vol\big(\Tcal^{\loc}_{\bx,\by}(1,\alpha)\big)+O\Big(\frac{W}{t}\Big)\Bigg),
    \end{equation*}
    which shows that
    \begin{equation}\label{equation:S_loc formula}
        S_{\bx,\by}^{\loc}(A,B;l)=2\frac{A^2}{l^2}\frac{1}{\det\Big(\Gamma^{\loc}_{\bx,\by}\Big(\frac{W}{\gcd(W,l)}\Big)\Big)}\Bigg(\vol\big(\Tcal^{\loc}_{\bx,\by}(1,\alpha)\big)+O\Big(\frac{Wl}{A}\Big)\Bigg)+O\Big(\frac{W^3}{A}\Big).
    \end{equation}
    Recalling Definition (\ref{def:Jcal volume}), we see that the volume in the display above is given by
    \begin{equation*}
        \vol\big(\Tcal_{\bx,\by}^{\loc}(1,\alpha)\big)=\Jcal\Big(2\alpha\frac{\bx}{\|\bx\|},2\alpha\frac{\by}{\|\by\|}\Big).
    \end{equation*}
    Now Lemma \ref{lemma:Jcal volume formula} computes that volume in terms of $\alpha$ and $\Delta(\bx,\by)$. It is
    \begin{equation}\label{equation:T^loc volume formula}
        \vol\big(\Tcal_{\bx,\by}^{\loc}(1,\alpha)\big)= \frac{\pi}{2}\frac{\Delta(\bx,\by}{\alpha^2}\Big(1+O\Big(\min\Big\{1,\frac{\Delta(\bx,\by)^2}{\alpha^2}\Big\}\Big)\Big).
    \end{equation}
    For the determinant in (\ref{equation:S_loc formula}), we may use Lemma \ref{lemma:determinants of Lambda intersections} to obtain
    \begin{equation}\label{equation:Gamma_loc determinant formula}
        \det\Big(\Gamma^{\loc}_{\bx,\by}\Big(\frac{W}{\gcd(W,l)}\Big)\Big)^{-1}=\frac{\gcd(W,l)^2\Gcal(\bx,\by)}{W^2}\big(1+O\big(\textbf{1}_{\Gcal(\bx,\by)\nmid W/\text{rad}(W)}\big)\big).
    \end{equation}
    Using that $\Delta(\bx,\by)\geq 1$ and recalling the definition (\ref{def:Ecal_x,y}) of $\Ecal_{\bx,\by}(B)$, we combine (\ref{equation:S_loc formula}) with (\ref{equation:T^loc volume formula}) and (\ref{equation:Gamma_loc determinant formula}) to get
    \begin{equation*}
        S_{\bx,\by}^{\loc}(A,B;l)=\pi\frac{A^2\gcd(W,l)^2}{\alpha^2W^2l^2}\Delta(\bx,\by)\Gcal(\bx,\by)\Big(1+O\Big(\Ecal_{\bx,\by}(B)+\frac{\alpha^2Wl}{A}\Big)\Big)+O\Big(\frac{W^3}{A}\Big).
    \end{equation*}
    By (\ref{W_bound}) we have $W\ll (\log B)^3\ll (\log A)^3$. We put this back into (\ref{equation:S_2^star as S^loc representation}) together with the fact that $\Gcal(\bx,\by),\Delta(\bx,\by) \geq 1$ and obtain
    \begin{align*}
        \Scal_2^*(\Gamma^{\loc}_{\bx,\by}(W),\Tcal^{\loc}_{\bx,\by}(A,\alpha))=&
        \frac{\pi A^2}{\alpha^2W^2}\Delta(\bx,\by)\Gcal(\bx,\by)\\
        &\times\Bigg(\sum_{l\leq A/W}\mu(l)\frac{\gcd(W,l)^2}{l^4}+O\Big(\Ecal_{\bx,\by}(B)+\frac{\alpha^2W}{A}\Big)\Bigg).
    \end{align*}
    The sum over $l\leq A/W$ is part of an absolutely convergent series. In particular, we can complete the sum and use that $\gcd(W,l)=l$ for all $l\leq w$. Therefore we have
    \begin{equation*}
        \sum_{l\leq A/W}\mu(l)\frac{\gcd(W,l)^2}{l^4}=\frac{1}{\zeta(2)}+O\Big(\frac{W}{A}+\frac{1}{w}\Big).
    \end{equation*}
    For the error term, note that $w\ll A/\alpha^2W$. Again recalling the definition (\ref{def:F(B)}) of $F(B)$ and putting everything back into (\ref{equation:D^loc representation as S_2^* sum}), we finally arrive at
    \begin{equation*}
        D^{\loc}(A,B)=\frac{\pi}{\zeta(2)}A^2E(B)\Big(1+O\Big(\frac{1}{w}+\frac{F(B)}{B^6}\Big)\Big).
    \end{equation*}
    Now using Lemma \ref{lemma:F(B) saving} and recalling that $w\gg (\log\log B)^{1/2}$ completes the proof.
\end{proof}

\subsection{First moment bounds}\label{section:first moment bounds}$~$

In this section we use the results of the previous subsection, together with first moment estimates to complete the proof of Proposition \ref{prop:variance_upper_bound}. Recall the definitions (\ref{def:Delta^mix and ^loc}) of $\Delta^{\mix}_{\bfa}(B)$ and $\Delta_{\bfa}^{\loc}(B)$. We let
\begin{equation}\label{def:K(A,B)}
    K(A,B)=\sum_{\substack{\bfa\in\ZZ^4_{\prim}\\\|\bfa\|\leq A}}\big((\log B)^4N_{\bfa}(B)+\Delta^{\mix}_{\bfa}(B)+\Delta^{\loc}_{\bfa}(B)\big).
\end{equation}
Further recall the definition (\ref{defn:N^loc}) of $N_{\bfa}^{\loc}(B)$ and consider
\begin{equation}\label{eq:V(A,B) definition}
    V(A,B)=\sum_{\substack{\bfa\in\ZZ^4_{\prim}\\\|\bfa\|\leq A}}\big|N_{\bfa}(B)-N^{\loc}_{\bfa}(B)\big|^2.
\end{equation}
Opening the square, we obtain
\begin{equation*}
    V(A,B)=D(A,B)-2D^{\mix}(A,B)+D^{\loc}(A,B)+K(A,B).
\end{equation*}
Now by the lemmas \ref{lemma:D(A,B) asymptotic formula}, \ref{lemma:D^mix asymptotic formula} and $\ref{lemma:D^loc asymptotic formula}$, for $A\geq B(\log B)^2$ there is a $1/2>\delta>0$ such that we have
\begin{equation*}
    D(A,B)-2D^{\mix}(A,B)+D^{\loc}(A,B)\ll \frac{A^2B^6}{(\log\log B)^{1/2-\delta}}.
\end{equation*}
To establish Proposition $\ref{prop:variance_upper_bound}$, we need to give an upper bound for $K(A,B)$ for $A$ and $B$ in a suitable range.
\begin{lemma}
    Let $B\leq A\leq B^3(\log B)^{-4}(\log\log B)^{-1}$. Then we have that
    \begin{equation*}
        K(A,B)\ll \frac{A^2B^6}{(\log\log B)}.
    \end{equation*}
\end{lemma}
\begin{proof}
    We split the sum in (\ref{def:K(A,B)}) into three parts according to $N_{\bfa}(B)$, $\Delta_{\bfa}^{\mix}(B)$ and $\Delta_{\bfa}^{\loc}(B)$ and estimate each one individually. By the definition (\ref{defn:N_a(B)}) of $N_{\bfa}(B)$ we have
    \begin{equation*}
        \sum_{\substack{\bfa\in\ZZ^4_{\prim}\\\|\bfa\|\leq A}}N_{\bfa}(B)=(\log B)^4\sum_{\bx\in\Pri(B)}\#\big\{\bfa\in\ZZ^4_{\prim}\cap\Lambda_{\bx}:\|\bfa\|\leq A\big\}.
    \end{equation*}
    By Lemma \ref{lemma:determinant of Lambda}, the determinant of $\Lambda_{\bx}$ is given by $\|\bx\|$. Note that $\|\bx\|\leq B\leq A$ and hence we are in position to apply Lemma \ref{Lemma:General point count}. We disregard the error term in said Lemma to just give an upper bound in the form of
    \begin{equation*}
        \#\big\{\bfa\in\ZZ^4_{\prim}\cap\Lambda_{\bx}:\|\bfa\|\leq A\big\}\ll\frac{A^3}{\|\bx\|}.
    \end{equation*}
    Now we sum this over $\bx\in\Pri(B)$ and obtain
    \begin{equation*}
         \sum_{\substack{\bfa\in\ZZ^4_{\prim}\\\|\bfa\|\leq A}}N_{\bfa}(B)\ll A^3(\log B)^4\sum_{\bx\in\Pri(B)}\frac{1}{\|\bx\|}.
    \end{equation*}
    To produce a bound for the remaining sum, we can apply partial summation together with the prime number theorem. We arrive at the estimate
    \begin{equation}\label{1st moment:N_a(B) bound}
        \sum_{\substack{\bfa\in\ZZ^4_{\prim}\\\|\bfa\|\leq A}}(\log B)^4N_{\bfa}(B)\ll(\log B)^8A^3\frac{B^3}{(\log B)^4}
    \end{equation}
    To deal with the sum over $\Delta_{\bfa}^{\mix}(B)$, we apply partial summation to the sum over $\|\bfa\|\leq A$. Thereafter we can continue as in the previous case to obtain
    \begin{equation}\label{1st moment:Delta^mix bound}
        \sum_{\substack{\bfa\in\ZZ^4_{\prim}\\\|\bfa\|\leq A}}\Delta_{\bfa}^{\mix}(B)\ll \alpha WA^2(\log B)^8\sum_{\bx\in\Pri(B)}\frac{1}{\|\bx\|^2}
        \ll \alpha W A^2B^2(\log B)^4.
    \end{equation}
    Similarly, for $\Delta_{\bfa}^{\loc}(B)$, we get
    \begin{equation*}
        \sum_{\substack{\bfa\in\ZZ^4_{\prim}\\\|\bfa\|\leq A}}\Delta_{\bfa}^{\loc}(B)\ll \alpha^2W^2A^2(\log B)^8\sum_{\bx\in\Pri(B)}\frac{1}{\|\bx\|^2}\ll \alpha^2W^2A^2B^2(\log B)^4.
    \end{equation*}
    Combining the display above with (\ref{1st moment:Delta^mix bound}) and (\ref{1st moment:N_a(B) bound}), we find that $K(A,B)$ is bounded by
    \begin{equation*}
     A^2B^2(\log B)^4\big(AB+\alpha W+(\alpha W)^2\big).
    \end{equation*}
    Recall that $\alpha=\log B$ and $W\ll(\log B)^3$. Therefore the last two summands are small enough and we only need to pick $A$ small enough in terms of $B$ so that $AB\ll B^4(\log B)^{-4}(\log\log B)^{-1}$. A choice of
    \begin{equation*}
        A\leq \frac{B^3}{(\log B)^4(\log\log B)} 
    \end{equation*}
    is sufficient.
\end{proof}

\section{The local counting function}\label{section:N^loc_lower_bound}$~$
The goal of this section is to prove Proposition \ref{prop:number_of_bad_N^loc}. We define two local factors in the same way as done in \cite[Section 4]{Hol} and then compare the local counting function $N_{\bfa}^{\loc}(B)$ to these. Afterwards we establish lower bounds for each of them to obtain the proposition. To start, let $\gamma>0$ and $\bv,\bfa\in\RR^4$. We let
\begin{equation*}
    \tau(\bfa,\gamma)=\gamma\cdot\vol\big(\{\buld\in(\Ball_4(1)\cap\RR^4_+):\bfa\in\Ccal_{\buld}^{(\gamma)}\}\big).
\end{equation*}
For any integer $Q\geq 1$ and $\bfa\in\ZZ^4$, we let
\begin{equation}\label{def:Rfrak}
    \mathfrak{R}(Q)=\big((\ZZ/Q\ZZ)^*\big)^4,
\end{equation}
and
\begin{equation}\label{def:sigma(a,Q)}
    \sigma(\bfa,Q):=\frac{Q}{\phi(Q)^4}\cdot\#\{\textbf{b}\in\mathfrak{R}(Q):\,\bfa\in\Lambda^{(Q)}_{\textbf{b}}\}.
\end{equation}
Recalling the definitions (\ref{defn:alpha}) and (\ref{def:W}) of $\alpha$ and $W$, we then define
\begin{equation}\label{def:Singular Integral}
    \Jfrak_{\bfa}(B)=\tau(\bfa,\alpha),
\end{equation}
and
\begin{equation}\label{def:SSSingular series}
    \SSS_{\bfa}(B)=\sigma(\bfa,W).
\end{equation}
We will see that $\SSS_{\bfa}(B)$ behaves like a truncated singular series that one would obtain by an application of the circle method to the counting problem (\ref{Equation 1}). Accordingly, $\Jfrak_{\bfa}(B)$ behaves similar to the corresponding singular integral. We relate these two quantities to the local counting function $N_{\bfa}^{\loc}(B)$ with the following lemma.
\begin{lemma}\label{lemma:N^loc_and_SigmaJ_comparison}
Let $A\geq 2$, $B\geq 3$ and $\bfa\in\ZZ^4_{\prim}$ with $\|\bfa\|\leq A$. For all $C>0$ we have
\begin{equation*}
    \Sfrak_{\bfa}(B)\Jfrak_{\bfa}(B)\ll \frac{A}{B^3}N_{\bfa}^{\loc}(B)+\frac{1}{(\log B)^C}.
\end{equation*}
\end{lemma}
\begin{proof}
    Recall the definition (\ref{defn:N^loc}) of $N_{\bfa}^{\loc}(B)$. We have
    \begin{equation*}
        (\log B)^4\sum_{\substack{\bx\in\Pri((B))\\\bfa\in\Lambda_{\bx}^{(W)}\cap\Ccal_{\bx}^{(\alpha)}}}1\ll\frac{AB}{\alpha W}N_{\bfa}^{\loc}(B).
    \end{equation*}
    As done in \cite[Section 5]{BrLBSa}, we split the sum over $\bx\in\Pri(B)$ into residue classes modulo $W$ to obtain
    \begin{equation*}
        \sum_{\substack{\bx\in\Pri(B)\\\bfa\in\Lambda_{\bx}^{(W)} \cap\,\Ccal_{\bx}^{(\alpha)}}}1\geq \sum_{\substack{\bald\in\Rfrak(W)\\\bfa\in\Lambda_{\bald}^{(W)}}}\#\Big\{\bx\in\Pri(B)\,:\begin{array}{l}
         \bx\equiv\bald\,(\text{mod }W)  \\
         \bfa\in\Ccal_{\bx}^{(\alpha)}
    \end{array}\Big\}+O(W).
    \end{equation*}
    We let
    \begin{equation*}
        R_{\bfa,B}=\{\buld\in\Ball_4(B)\cap\RR^4_+:\bfa\in\Ccal_{\buld}^{(\alpha)}\},
    \end{equation*}
    and
    \begin{equation*}
        \Pri_{\buld}=\{\bx\in\Pri:\bx\equiv\buld\,(\text{ mod }W)\}.
    \end{equation*}
    Now the set in the sum above is the same as $R_{\bfa,B}\cap\Pri_{\buld}$, therefore we get that
    \begin{equation}
        \frac{AB}{\alpha W}N^{\loc}_{\bfa}(B)\gg (\log B)^4\sum_{\substack{\buld\in\Rfrak(W)\\\bfa\in\Lambda_{\buld}^{(W)}}}\#(R_{\bfa,B}\cap\Pri_{\buld})+O\big((\log B)^4W\big).
    \end{equation}
    We can now directly follow the proof of \cite[Lemma 4.2]{Hol}, by picking a real $3/5<\theta<1$ and cutting the region $R_{\bfa,B}$ into cubes of side length $B^{\theta/3}$. Then an application of a short-interval version of the Siegel-Walfisz theorem can yield a lower bound for $\#R_{\bfa,B}\cap \Pri_{\buld}$. This works exactly as in \cite[pp 25-27]{Hol} by putting $n+1$ in their setting to be $4$ in our setting. 
\end{proof}

Next we combine the previous lemma with lower bounds for $\SSS_{\bfa}(B)$ and $\Jfrak_{\bfa}(B)$ that we will obtain in the following subsections. Recall the definition of the set $\LL^{\loc}$ which contains all primitive integer tuples whose the coordinates are not all of the same sign and for which Equation (\ref{Equation 1}) has a solution in the reduced residues modulo every prime. Then we will prove the following.
\begin{lemma}\label{lemma:singular terms av lower bound}
    Let $B\geq 2$ and $A\geq B^2$. Let $C>0$ and $\kappa>0$, then we have that
    \begin{equation*}
        \#\big\{\bfa\in\LL^{\loc}(A):\,\SSS_{\bfa}(B)\Jfrak_{\bfa}(B)\leq C(\log\log A)^{-\kappa}\big\}\ll \frac{A^4}{(\log\log A)^{\kappa/3}},
    \end{equation*}
    where the implied constant depends at most on $C$ and $\kappa$.
\end{lemma}
This lemma, together with Lemma \ref{lemma:N^loc_and_SigmaJ_comparison}, directly yields Proposition \ref{prop:number_of_bad_N^loc}. We use the following subsection to establish a lower bound for the factor $\SSS_{\bfa}(B)$ and then use Subsection \ref{subsection:sing_series_lower_bound} to complete the prove of the lemma above by giving an average bound on $\Jfrak_{\bfa}(B)$.

\subsection{The singular series is bounded from below.}$~$
Let $Q\geq 1$ be an integer and $\bfa\in\ZZ^4_{\prim}$. Recall the introduction (\ref{def:Rfrak}) of the set $\Rfrak(Q)$ at the beginning of this section. We let
\begin{equation}\label{def:rho_a(Q)}
    \rho_{\bfa}(Q)=\#\{\buld\in\Rfrak(Q):\langle\bfa,\buld\rangle\equiv 0\,(\text{ mod }Q)\}.
\end{equation}
By the Chinese remainder theorem the function above is multiplicative in $Q$. Therefore we have that
\begin{equation}\label{Equation:sigma-rho_relation}
    \sigma(\bfa,Q)=\frac{Q}{\phi(Q)^4}\rho_{\bfa}(Q)=\prod_{p^l\|Q}\frac{p^l}{\phi(p^l)^4}\rho_{\bfa}(p^l).
\end{equation}
Given integers $q\geq 1$ and $r$, we introduce the Ramanujan sum
\begin{equation*}
    c_q(r)=\sideset{}{^*}\sum_{a\,(q)}e(ar/q),
\end{equation*}
so that, writing $\bfa=(a_1,\dots,a_4)$, it is
\begin{align}\label{Equation:rho_ramanujan_darstellung}
    \rho_{\bfa}(q)&= \frac{1}{q}\sum_{\bald\in\Rfrak(q)}\sum_{c\,(q)} e\big((a_1b_1+a_2b_2+\dots+a_4b_4)c/q\big)\\
    &=\frac{1}{q}\sum_{c\,(q)}c_q(a_1c)c_q(a_2c)\dots c_q(a_4c).\nonumber
\end{align}
By (\ref{Equation:sigma-rho_relation}), we only need to consider the case $q=p^l$ for primes $p$ and integers $l\geq 1$. For this, we have the following well known lemma.
\begin{lemma}
    Let $p$ be a prime and $l\geq 1,\, r$ be integers. Then we have that
    \begin{equation*}
        c_{p^l}(r)=\begin{cases}
            0 &\text{if }p^{l-1}\nmid r,\\
            -p^{l-1} &\text{if }p^{l-1}|r,\,p^l\nmid r,\\
            \phi(p^l)&\text{if }p^l|r.
        \end{cases}
    \end{equation*}
\end{lemma}
Note that for any $\bfa=(a_1,\dots,a_4)\in\ZZ^4_{\prim}$ and any prime $p$, there is some $1\leq i\leq 4$ such that $p$ and $a_i$ are coprime. Thus, after reducing (\ref{Equation:rho_ramanujan_darstellung}) to prime powers $p^l$, we only need to consider residue classes $c\mod p^l$ where $p^{l-1}|c$. With the substitution $c=c'p^{l-1}$ we get
\begin{align*}
    \rho_{\bfa}(p^l)&=\frac{1}{p^l}\sum_{\substack{c\,(p^l)\\p^{l-1}|c}}c_{p^l}(a_1c)c_{p^l}(a_2c)\dots c_{p^l}(a_4c)\\
    &=\frac{1}{p^l}\sum_{c'\,(p)}c_{p^l}(a_1p^{l-1}c')c_{p^l}(a_2p^{l-1}c')\dots c_{p^l}(a_4p^{l-1}c').
\end{align*}
While the summand with $c'=0$ evaluates to $\phi(p^l)^4$, for $c'\neq 0$ each factor yields either $-p^{l-1}$ or $\phi(p^l)$ according to whether $a_i$ is coprime to $p$ or not. Hence let $\lambda$ be the number of $1\leq i\leq 4$ such that $p$ divides $a_i$. With this we have
\begin{equation}\label{rho_a(Q)(p^l) formula}
    \rho_{\bfa}(p^l)=\frac{1}{p^l}\Big(\phi(p^l)^4+(p-1)\phi(p^l)^{\lambda}(-p^{l-1})^{4-\lambda}\Big).
\end{equation}
Relating this back to $\sigma(\bfa,p^l)$ with (\ref{Equation:sigma-rho_relation}), we get that
\begin{align}\label{sigma_p^l_formula}
    \sigma(\bfa,p^l)&=\frac{1}{\phi(p^l)^4}\Big(\phi(p^l)^4+(p-1)\phi(p^l)^{\lambda}(-p^{l-1})^{4-\lambda}\Big)\\
    &= 1+(p-1)\Big(\frac{-p^{l-1}}{\phi(p^l)}\Big)^{4-\lambda}.\nonumber
\end{align}
At this point it is sensible to make a distinction between primes $p$ according to their coprimality to $\bfa$. We let
\begin{equation*}
    \Pri_G=\{p\in\Pri:\gcd(p;a_i)=1\text{ for }1\leq i\leq 4\}\quad\text{and}\quad \Pri_B=\Pri\setminus\Pri_G.
\end{equation*}
Now for $p\in\Pri_G$ with $p\geq 2$ it is $\lambda=0$ and thus
\begin{equation*}
    \sigma(\bfa,p^l))=1+\frac{1}{(p-1)^3}>1.
\end{equation*}
This is a factor of an absolutely convergent Euler product and therefore products over such factors are bounded from above and below. For $p\in\Pri_B$, we have that $\lambda\in\{1,2,3\}$. Whenever $\lambda=3$, by (\ref{sigma_p^l_formula}) it is
\begin{equation*}
    \sigma(\bfa,p^l)=0.
\end{equation*}
If we have that $\lambda=2$, then we get
\begin{equation*}
    \sigma(\bfa,p^l)=1+\frac{1}{(p-1)}>1
\end{equation*}
Since this factor is large enough and there is finitely many of them for a fixed $\bfa$, we can dismiss this case for a lower bound. In the last case, when $\lambda=1$, we make a distinction between the prime $2$ and primes $p\geq 3$. For the former we have that
\begin{equation*}
    \sigma(\bfa,2^l)=1-1=0.
\end{equation*}
For the remaining $p\geq 3$, it is
\begin{equation*}
    \sigma(\bfa,p^l)=1-\frac{1}{(p-1)^2}\geq 1-\frac{4}{p^2},
\end{equation*}
which is a factor of an absolutely convergent Euler product and thus there exists some $C>0$ independent of $\bfa$ such that products over these factors are $\geq C$.

To state the result of this subsection, we quickly relate this to the local solvability of Equation (\ref{Equation 1}). For a prime $p$, if there is a solution $\bald=(b_1,\dots,b_4)\in\Rfrak(p^l)$, for some $l\geq 1$, to
\begin{equation*}
    a_1b_1+a_2b_2+a_3b_3+a_4b_4\equiv 0\,(\text{ mod }p^l),
\end{equation*}
it is $\rho_{\bfa}(p^l)\neq 0$ by definition and thus also $\sigma(\bfa,p^l)\neq 0$. This suffices to obtain the following lemma.
\begin{lemma}
    Let $\bfa\in\LL^{\loc}$ and $\SSS_{\bfa}(B)$ as defined in (\ref{def:SSSingular series}). There exists some $C>0$ such that we have
    \begin{equation*}
        \SSS_{\bfa}(B)\geq C >0.
    \end{equation*}
\end{lemma}

\subsection{The singular integral is bounded from below}\label{subsection:sing_series_lower_bound}$~$
In this section we prove Lemma \ref{lemma:singular terms av lower bound}. We start by proving a point-wise lower bound for $\Jfrak_{\bfa}(B)$ which might be small in certain cases but is large enough on average to obtain the lemma. During this section, for a vector $\bx=(x_1,\dots,x_N)\in\RR^N$, we let
\begin{equation*}
    \|\bx\|_{\min}=\min_{i\in\{1,\dots,N\}}|x_i|.
\end{equation*}
This is not a norm on $\RR^N$. We obtain the lower bound by finding some region that is contained in the set
\begin{equation*}
    \big\{\buld\in\big(\Ball_4(1)\cap\RR^4_+\big):\bfa\in\Ccal_{\buld}^{(\gamma)}\big\},
\end{equation*}
for which we can give an estimate on its volume. This then yields a lower bound for $\Jfrak_{\bfa}(B)$. For a nonzero-integer vector $\bfa$ we let
\begin{equation*}
    \delta=\frac{\|\bfa\|_{\min}}{\|\bfa\|_{\infty}}.
\end{equation*}
We then have the following lemma.
\begin{lemma}
    Let $\bfa=(a_1,\dots,a_4)\in(\ZZ\setminus\{0\})^4$ where the $a_i$ are not all of the same sign. Then there exists a $\bx=(x_1,\dots,x_4)\in\RR^4_+$ with $\|\bx\|=1$ such that $\langle\bfa,\bx\rangle=0$ and $x_i\gg\delta$. The implied constant does not depend on $\bfa$.
\end{lemma}

\begin{proof}
    We start by reorganizing the entries of $\bfa$ such that there is a $1\leq r\leq 3$ so that $a_1,\dots,a_r$ are positive and $a_{r+1},\dots,a_4$ are negative. Furthermore, let $1\leq j\leq 4$ be an integer so that $|a_j|=\|\bfa\|_{\infty}$. If we have that $j\leq r$, we can pick $x_1,\dots,x_r$ equal to $\delta$ and find that
    \begin{equation*}
        r\|\bfa\|_{\infty}\geq a_1x_1+\dots+a_rx_r\geq\|\bfa\|_{\min}.
    \end{equation*}
    Now for $r+1\leq i\leq 4$, we can pick $x_i=\big((4-r)|a_i|\big)^{-1}\big(\sum_1^ra_ix_i\big)$. Then these last $x_i$ are at least of size $\delta/4>0$ and with $\bx=(x_1,\dots,x_4)$ we have found a vector such that $\langle\bfa,\bx\rangle=0$. Moreover, it is $\|\bx\|\leq 4\sqrt{4}$ and thus there is a $c>1$ such that the vector $\by=c(4\sqrt{4})^{-1}\bx$ satisfies the statement of the lemma.

    In the case where there are only $j>r$ such that $|a_j|=\|\bfa\|_{\infty}$, we do the same argument with flipped signs.
\end{proof}
We observe that it is
\begin{equation*}
    \vol\big\{\buld\in\big(\Ball_4(1)\cap\RR^4_+\big):|\langle\bfa,\buld\rangle|\leq\frac{\|\bfa\|\|\buld\|}{2\gamma}\big\} \geq \vol\big\{\buld\in\big(\big(\Ball_4(1)\setminus\Ball_4(\frac{1}{4})\big)\cap\RR^4_+\big):|\langle\bfa,\buld\rangle|\leq \frac{\|\bfa\|\|\buld\|}{2\gamma}\big\}.
\end{equation*}
Now the lemma supplies a $1\geq c>0$ and $\bx=(x_1,\dots,x_4)$ with $\|\bx\|=1/2$, $\langle\bfa,\bx\rangle=0$ and $x_i\geq c\delta$. If a $\bvld\in\RR^4$ satisfies $\|\bvld-\bx\|\leq c\delta/4$, then it is $1/4\leq \|\bvld\|\leq 3/4$ and $\langle\bfa,\bvld-\bx\rangle=\langle\bfa,\bvld\rangle$. In that case we also have that $|v_i-x_i|\leq c\delta/4$ and thus $v_i\geq c\delta/4$. This shows that the volume above is bounded from below by
\begin{equation*}
    \vol\big\{\buld\in\Ball_4(c\delta/4):|\langle\bfa,\buld\rangle|\leq \frac{\|\bfa\|}{8\gamma}\big\}.
\end{equation*}
Next, letting $\Tilde{\bfa}=\bfa/\|\bfa\|$, this volume is equal to
\begin{equation}\label{equation:vol_mit_cdelta}
    \vol\big\{\buld\in\Ball_4(c\delta/4):|\langle\Tilde{\bfa},\buld\rangle|\leq (8\gamma)^{-1}\big\}.
\end{equation}
We further reduce the volume so that we can write down a parametrization allowing us to give an explicit lower bound. Without loss of generality we may assume that $\|\bfa\|_{\infty}=|a_1|$. For a $\bw\in\RR^3$, we let
\begin{equation*}
    W_{\Tilde{\bfa},\gamma}(\bw)=\vol\big\{v_0\in[-3c\delta/8,3c\delta/8]:|\langle\Tilde{\bfa},(v_0,\bw)\rangle|\leq (8\gamma)^{-1}\big\},
\end{equation*}
so that the volume in (\ref{equation:vol_mit_cdelta}) is bounded from below by
\begin{equation}\label{equation:vol_mit_integral_und_W}
    \int_{\Ball_3(c\delta/8)}W_{\Tilde{\bfa},\gamma}(\bw)\,d\bw.
\end{equation}
Given a $\bw=(w_2,w_3,w_4)\in\Ball_3(c\delta/8)$, let $v_1=\Tilde{a}_2w_2+\dots+\Tilde{a}_4w_4$, so that  the $v_0$ in $W_{\Tilde{\bfa},\gamma}(\bw)$ are given by the inequalities
\begin{equation*}
    |v_0|\leq\frac{3c\delta}{8},\quad\text{and}\quad\Big|v_0-\frac{v_1}{\Tilde{a}_1}\Big|\leq \frac{1}{8\gamma}.
\end{equation*}
Note that since we assumed $\|\bfa\|_{\infty}=|a_1|$, it is $|v_1/\Tilde{a}_1|\leq 3c\delta/8$ and thus this set is non empty. Now if $\delta\gg\gamma^{-1}$, then the size of this set is bounded from below by $\gamma^{-1}$ since an interval of length $(16\gamma)^{-1}$ is included in the intersection of the intervals given by the two inequalities above. When on the other hand we have $\gamma^{-1}\gg\delta$, then an interval of length $\gg\delta$ is included. Therefore we have that
\begin{equation*}
    W_{\Tilde{\bfa},\gamma}(\bw)\gg\min\Big\{\delta,\frac{1}{\gamma}\Big\}.
\end{equation*}
Putting this back into (\ref{equation:vol_mit_integral_und_W}) and recalling the chain of inequalities, we find that we have shown
\begin{equation}\label{lower_bound:singular_integral}
    \Jfrak_{\bfa}(B)\gg \gamma\delta^3\min\Big\{\delta,\frac{1}{\gamma}\Big\}\geq \delta^3\min\big\{\gamma\delta,1\big\},
\end{equation}
where the implied constant does not depend on $\bfa$ or $\gamma$. Note that depending on $\bfa$, this lower bound can get as small as $\|\bfa\|^{-3}$. However, in that case $\bfa$ must have some small coefficients, which does not happen often. We put this observation into the following lemma and then Lemma \ref{lemma:singular terms av lower bound} will follow immediately.
\begin{lemma}\label{lemma:Jfrak_average_lower_bound}
    Let $A$ and $B$ be sufficiently large real numbers such that $A\leq B^3$. Let $C>0$ and $\kappa>0$, then we have that
    \begin{equation*}
        \#\Big\{\bfa\in\LL^{\loc}(A):\,\Jfrak_{\bfa}(B)\leq\frac{C}{(\log\log B)^{\kappa}}\Big\}\ll\frac{A^4}{(\log\log A)^{\kappa/3}}.
    \end{equation*}
\end{lemma}

\begin{proof}
    Given $C>0$ and $\kappa>0$, suppose that
    \begin{equation*}
        \Jfrak_{\bfa}(B)\leq \frac{C}{(\log\log B)^{\kappa}}.
    \end{equation*}
    By (\ref{lower_bound:singular_integral}), there exists some $\Tilde{C}>0$ such that we have
    \begin{equation*}
        \delta^3\min\{\gamma\delta,1\}\leq\frac{\Tilde{C}}{(\log\log B)^{\kappa}}.
    \end{equation*}
    If $\min\{\gamma\delta,1\}=\gamma\delta$, then this implies that $\|\bfa\|_{\min}\leq\|\bfa\|_{\infty}(\log B)^{-1}\ll A(\log A)^{-1}$ since $A\leq B^3$. There are at most $O(A^4(\log A)^{-1})$ many of these tuples. If on the other hand $\min\{\gamma\delta,1\}=1$, then we must have that
    \begin{equation*}
        \Big(\frac{\|\bfa\|_{\min}}{\|\bfa\|_{\infty}}\Big)^3\leq \frac{\Tilde{C}}{(\log\log B)^{\kappa}}.
    \end{equation*}
    We can apply the same argument as above to conclude that there are at most $O(A^4(\log\log A)^{-\kappa/3})$ many of these tuples.
\end{proof}

\section{The density of locally solvable linear forms}\label{section:density_theorem_proof}$~$

In this section we prove Theorem \ref{theorem:L^loc_density_estimate}. The strategy is straightforward; we pick a suitable subset of $\LL^{\loc}(A)$ whose cardinality we can compute and find that it is large enough. The set which is presented below is not directly a subset of $\LL^{\loc}(A)$, but there is a bijection to one. Therefore it is sufficient to consider that set for our purposes.
\begin{defn}\label{Definition:Subset}
    Let $A\geq 1$ a real number. We let $\LL'(A)$ be the set of all $\bfa=(a_1,\dots,a_4)\in\NN^4$ subject to the following constraints:
    \begin{enumerate}
        \item $|\bfa|\leq A$,
        \item all of the $a_i$ are odd,
        \item we have that $(a_1a_2;a_3)=(a_1a_2;a_4)=1$.
    \end{enumerate}
\end{defn}
Next, we convince ourselves that this is indeed in bijection to a subset of $\LL^{\loc}(A)$. First, the condition (3) in \ref{Definition:Subset} yields that $\LL'(A)$ is a subset of $\ZZ^4_{\prim}$. Also since we have that $|\bx|\leq\|\bx\|$ for all $\bx\in\RR^4$, the condition (1) implies that this is a subset of $\LL(A)$. We check that for each $\bfa=(a_1,\dots,a_4)\in\LL'(A)$ and each prime $p$, the equation
\begin{equation}\label{Equation:local}
    a_1x_1+a_2x_2+a_3x_3+a_4x_4\equiv 0 \quad(\mod p)
\end{equation}
has solutions with $x_i\in(\ZZ/p\ZZ)^*$. The condition (2) in \ref{Definition:Subset} yields that $a_i\equiv 1(\mod 2)$, so that Equation (\ref{Equation:local}) is trivially solvable for $p=2$. For $p>2$, we use Condition (3) to find that both $a_3$ and $a_4$ are coprime to $a_1$ and $a_2$ each. In particular we have that $\gcd(a_i;a_j;a_k)=1$ for any $i,j,k$ pairwise distinct. Now we can use results from Section \ref{section:N^loc_lower_bound}. Recalling the definition (\ref{def:rho_a(Q)}) of $\rho_{\bfa}(Q)$ for a positive integer $Q\geq 2$ together with Equation (\ref{rho_a(Q)(p^l) formula}) we find that there are solutions to Equation (\ref{Equation:local}) for each prime $p>2$. \\
Finally, the definition of $\LL^{\loc}(A)$ requires the $a_i$ to not all be of the same sign. However, switching a sign in $\bfa\in\LL'(A)$ does not change any of the conditions (1), (2) or (3) in \ref{Definition:Subset}. Thus the set $\LL'(A)$ is in bijection to a subset of $\LL^{\loc}(A)$ by changing a sign of one of the coordinates.
The next proposition gives us an estimate on the cardinality of $\LL'(A)$, from which we can immediately conclude Theorem \ref{theorem:L^loc_density_estimate}.

\begin{proposition}\label{prop:LL'lowerbound}
    There exists a $C>0$ such that for $A>1$ we have
    \begin{equation*}
        \#\LL'(A)=CA^4+O\big(A^3(\log A)\big).
    \end{equation*}
\end{proposition}
To prove this result, we start by presenting two small lemmas that simplify some computations for us.
\begin{lemma}\label{lemma:coprime sum}
    Let $q\in\NN$. For $X\geq 1$ and $\epsilon>0$ we have that
    \begin{equation*}
        \sum_{\substack{n\leq X\\(n;q)=1}}1=\frac{\phi(q)}{q}X+O(q^{\epsilon}).
    \end{equation*}
\end{lemma}
The proof of this lemma is a straightforward computation and will be omitted.
\begin{lemma}\label{lemma:doublesum divided}
    Let $q\in\NN$. For $X\geq 1$ and $\epsilon>0$ we have that
    \begin{equation*}
        \sum_{\substack{n,m\leq X\\q|nm\\n,m\text{ odd}}}1=\begin{cases}
             \frac{\sum_{d|q}d\phi(q/d)}{q^2}X^2+O(Xq^{\epsilon}) &\text{if $q$ is odd,}\\
             0 &\text{if $q$ is even}.
        \end{cases}
    \end{equation*}
\end{lemma}
\begin{proof}
Let $n,m$ and $q$ be positive integers. Let $k=(q;n)$, we claim that $q/k$ dividing $m$ is equivalent to $q|nm$. The first statement clearly implies the second one. For the reverse implication, one can consider the $p$-adic valuation of $p/k$ for each prime $p$. We have $v_p(q)\leq v_p(m)+v_p(n)$ and by the definition of $k$ we also get that $v_p(q/k)=v_p(q)-\min\{v_p(q),v_p(n)\}$. Now $v_p(q/k)$ takes either the value $v_p(q)-v_p(q)=0$, or $v_p(q)-v_p(n)$, both of which are smaller than or equal to $v_p(m)$ for each prime $p$. We now use this to split the sums over $n$ and $m$ so that we can execute them separately. It is
    \begin{equation*}
        \sum_{\substack{n,m\leq X\\q|nm\\n,m\text{ odd}}}1= 
        \sum_{\substack{n\leq X\\ n\text{ odd}}}\sum_{\substack{m\leq X\\q/(q;n)|m\\ m\text{ odd}}}1
    \end{equation*}
    If $q$ is even, then the sums evaluate to 0, so we only need to consider odd $q$. In that case the last sum evaluates to $[X(q;n)/q]=X(q;n)/q+O(1)$.\\
    Now we are left to compute the sum
    \begin{equation*}
        \sum_{\substack{n\leq X\\ n\text{ odd}}}(q;n).
    \end{equation*}
    By letting $d=(q;n)$, we can factor out the gcd so that we can invoke Lemma \ref{lemma:coprime sum}. In total we get
    \begin{align*}
        \sum_{\substack{n,m\leq X\\q|nm\\n,m\text{ odd}}}1&=
        \frac{X}{q}\sum_{d|q}d\sum_{\substack{n\leq X/d\\(q/d,n)=1\\ n\text{ odd}}}1+O\big(X\big)\\
        &=\frac{X^2}{q^2}\sum_{d|q}d\phi(q/d) +O\big(Xq^{\epsilon}\big).
    \end{align*}
\end{proof}
We can now turn our attention to computing $\#\LL'(A)$. Given positive integers $n$ and $m$ and a real $X\geq 1$, we let
\begin{equation*}
    \Ncal(X;n,m)=\big\{(a,b)\in\NN^2:\,a,b\leq X, (nm;a)=(nm;b)=1\big\}.
\end{equation*}
Note that if $n$ and $m$ are odd integers and we have that $(nm;a)=1$ for some integer $a$, then saying that $a$ is odd is equivalent to the condition that $(2nm;a)=1$. Therefore we can write
\begin{equation}\label{eq:LL'formel}
    \#\LL'(A)=\sum_{\substack{a_1\leq A\\a_1\text{ odd}}}\sum_{\substack{a_2\leq A\\a_2\text{ odd}}}\#\Ncal(A;2a_1,2a_2).
\end{equation}
Now we use Lemma \ref{lemma:coprime sum} twice to obtain
\begin{align*}
    \#\Ncal(A,2a_1,2a_2)&=\Big(\frac{\phi(2a_1a_2)}{2a_1a_2}A+O\big((a_1a_2)^{\epsilon}\big)\Big)^2\\
    &=\Big(\frac{\phi(2a_1a_2)}{2a_1a_2}\Big)^2A^2+O\big(A(a_1a_2)^{\epsilon}\big)
\end{align*}
for all $\epsilon>0$. When putting this back into (\ref{eq:LL'formel}), we note that $\phi(2a_1a_2)=\phi(a_1a_2)$ since $a_1a_2$ is odd. Therefore we get
\begin{equation}\label{eq:LL'formel2}
    \#\LL'(A)=\frac{A^2}{4}\sum_{\substack{a_1\leq A\\a_1\text{ odd}}}\sum_{\substack{a_2\leq A\\a_2\text{ odd}}}\Big(\frac{\phi(a_1a_2)}{a_1a_2}\Big)^2\,+O\big(A^{3+\epsilon}\big).
\end{equation}
Before we try to evaluate the remaining sums, we proceed by giving a weak lower bound on $\#\LL'(A)$. This will be helpful to establish the non vanishing of the constant featured in Proposition \ref{prop:LL'lowerbound}. We insert the well known estimate $\phi(n)\gg n/\log\log n$ for $n\geq 3$ into (\ref{eq:LL'formel2}) to obtain
\begin{equation}\label{eq:LL'trivial lower bound}
    \#\LL'(A)\gg A^2\sum_{\substack{3\leq a_1\leq A\\a_1\text{ odd}}}\sum_{\substack{3\leq a_2\leq A\\a_2\text{ odd}}}\frac{1}{\log\log A}\gg \frac{A^4}{\log\log A}.
\end{equation}
Now we return to our analysis of the main term in (\ref{eq:LL'formel2}). We remove $\phi(a_1a_2)$ by writing it as a convolution of the Möbius function with the identity on $\NN$. This yields
\begin{align*}
     \sum_{\substack{a_1\leq A\\a_1\text{ odd}}}\sum_{\substack{a_2\leq A\\a_2\text{ odd}}}\Big(\frac{\phi(a_1a_2)}{a_1a_2}\Big)^2&=\sum_{\substack{a_1\leq A\\a_1\text{ odd}}}\sum_{\substack{a_2\leq A\\a_2\text{ odd}}}\sum_{d|a_1a_2}\frac{\mu(d)}{d}\sum_{k|a_1a_2}\frac{\mu(k)}{k}\\
     &=\sum_{\substack{d\leq A\\d\text{ odd}}}\frac{\mu(d)}{d}\sum_{\substack{k\leq A\\k\text{ odd}}}\frac{\mu(k)}{k}\sum_{\substack{a_1,a_2\leq A\\ a_1,a_2\text{odd}\\dk/(d;k)|a_1a_2}}1,
\end{align*}
where we have used the equivalence $d,k|a_1a_2\Leftrightarrow dk/(d;k)|a_1a_2$. We are now in a situation where we can apply Lemma \ref{lemma:doublesum divided}, resulting in
\begin{align}\label{LL'formel3}
    \#\LL'(A)=\frac{A^4}{4}\sum_{\substack{d\leq A\\ d\text{ odd}}}\frac{\mu(d)}{d^3}\sum_{\substack{k\leq A\\ k\text{ odd}}}\frac{\mu(k)}{k^3}(d;k)^2\sum_{l|\frac{dk}{(d;k)}}l\phi\Big(\frac{dk}{(d;k)l}\Big)\,+O\big(A^{3+\epsilon}\big)
\end{align}
for all $\epsilon>0$. All that is left to do is to complete the remaining sums as part of absolutely convergent series and then confirm that the resulting constant is non-zero. For the convergence we forget about the fact that the sums in the display above run only over odd integers. Further, the situation is symmetric in $d$ and $k$. We estimate the innermost sum by $\sum_{l|dk}\frac{dk}{l}\phi(l)$, so that for $\epsilon>0$ we get
\begin{equation*}
    \sum_{d\leq A}\frac{1}{d^2}\sum_{k>A}\frac{(d;k)}{k^2}\sum_{l|dk}\frac{\phi(l)}{l}\ll\sum_{d\leq A}\frac{1}{d}\sum_{k>A}\frac{1}{k^2}\tau(dk)\ll A^{\epsilon-1},
\end{equation*}
Where we used that $\tau(dk)\ll (dk)^{\epsilon}$. Completing the remaining sum, we find that
\begin{equation*}
    \sum_{d>A}\frac{1}{d^2}\sum_{k>A}\frac{(d;k)}{k^2}\tau(dk)\ll \sum_{d>A}d^{-3/2+\epsilon}\sum_{k>A}k^{-3/2+\epsilon}\ll A^{\epsilon-1}.
\end{equation*}
Now combining this with (\ref{LL'formel3}), we find that in total we have
\begin{equation*}
    \#\LL'(A)=\frac{A^4}{4}\sum_{d\text{ odd}}\frac{\mu(d)}{d^2}\sum_{k\text{ odd}}\frac{\mu(k)}{k^2}(d;k)\sum_{l|\frac{dk}{(d;k)}}\frac{\phi(l)}{l}+O\big(A^{3+\epsilon}\big).
\end{equation*}
To complete the proof of Proposition \ref{prop:LL'lowerbound}, suppose that the constant given by the series in the display above vanishes. Then we would have $\#\LL'(A)\ll A^{3+\epsilon}$, contradicting (\ref{eq:LL'trivial lower bound}).

\bibliographystyle{alpha}
\bibliography{FormsIn4Variables}
\end{document}